\documentclass[reqno]{amsart}%
\usepackage{graphicx}
\usepackage{amsmath}
\usepackage{amsfonts}
\usepackage{amssymb}%
\providecommand{\U}[1]{\protect\rule{.1in}{.1in}}
\newtheorem{theorem}{Theorem}

\newtheorem{conjecture}[theorem]{Conjecture}

\newtheorem{definition}[theorem]{Definition}
\newtheorem{example}[theorem]{Example}

\newtheorem{idea}[theorem]{Idea}
\newtheorem{lemma}[theorem]{Lemma}
\newtheorem{notation}[theorem]{Notation}

\newtheorem{proposition}[theorem]{Proposition}
\newtheorem{remark}[theorem]{Remark}

\numberwithin{theorem}{section}
\numberwithin{equation}{section}
\begin{document}
\title[The heat flow conjecture for polynomials]{The heat flow conjecture for polynomials and random matrices}
\author{Brian C. Hall}
\address{Department of Mathematics, University of Notre Dame, Notre Dame, IN 46556, USA}
\email{bhall@nd.edu}
\author{Ching-Wei Ho}
\address{Institute of Mathematics, Academia Sinica, Taipei 10617, Taiwan}
\email{chwho@gate.sinica.edu.tw}
\thanks{Hall's research supported in part by a grant from the Simons Foundation}

\begin{abstract}
We study the evolution of the roots of a polynomial of degree $N$, when the
polynomial itself is evolving according to the heat flow. We propose a general
conjecture for the large-$N$ limit of this evolution. Specifically, we propose
(1) that the log potential of the limiting root distribution should evolve
according to a certain first-order, nonlinear PDE, and (2) that the limiting
root distribution at a general time should be the push-forward of the initial
distribution under a certain explicit transport map. These results should hold
for sufficiently small times, that is, until singularities begin to form.

We offer three lines of reasoning in support of our conjecture. First, from a
random matrix perspective, the conjecture is supported by a deformation
theorem for the second moment of the characteristic polynomial of certain
random matrix models. Second, from a dynamical systems perspective, the
conjecture is supported by the computation of the second derivative of the
roots with respect to time, which is formally small before singularities form.
Third, from a PDE\ perspective, the conjecture is supported by the exact
PDE\ satisfied by the log potential of the empirical root distribution of the
polynomial, which formally converges to the desired PDE\ as $N\rightarrow
\infty.$ We also present a \textquotedblleft multiplicative\textquotedblright%
\ version of the the conjecture, supported by similar arguments.

Finally, we verify rigorously that the conjectures hold at the level of the
holomorphic moments.

\end{abstract}
\maketitle
\tableofcontents

\section{Introduction}

\subsection{Heat flow on polynomials\label{introFlow.sec}}

Consider the standard heat equation for a function $u(x,t)$ on the real line,%
\[
\frac{\partial u}{\partial t}=\frac{1}{2}\frac{\partial^{2}u}{\partial x^{2}%
},
\]
with a polynomial initial condition,%
\[
u(x,0)=p(x).
\]
It is easily seen that the solution (computed, say, as the convolution of $p$
with a Gaussian of variance $t$) can be expressed as a terminating power
series,%
\[
u(x,t)=\sum_{k=0}^{\infty}\frac{1}{k!}\left(  \frac{t}{2}\right)  ^{k}\left(
\frac{d^{2}}{dx^{2}}\right)  ^{k}p(x),\quad x\in\mathbb{R},~t>0.
\]

Note that $u(x,t)$ is a polynomial in $x$ for each $t,$ with degree equal to
the degree of $p.$ The roots of this polynomial, however, are not necessarily
real, even if the roots of $p$ are real. It is therefore natural to
holomorphically extend both the initial condition and the solution in the
space variable from $\mathbb{R}$ to $\mathbb{C}$ and consider the function%
\begin{equation}
u(z,t)=\sum_{k=0}^{\infty}\frac{1}{k!}\left(  \frac{t}{2}\right)  ^{k}\left(
\frac{d^{2}}{dz^{2}}\right)  ^{k}p(z),\quad z\in\mathbb{C},~t>0,
\label{complexSeries}%
\end{equation}
where $d/dz$ is the complex derivative of a holomorphic function.

Once we have this (terminating) power series representation of the solution,
we note that the formula on the right-hand side of (\ref{complexSeries}) makes
sense with $t$ replaced by an arbitrary complex number. Furthermore, in the
analysis of the large-$N$ limit, it is convenient to scale the heat equation
in a way that depends on the degree $N$ of the initial condition. We therefore
define the time-$\tau$ \textbf{heat flow} operator $\exp\left(  \frac{\tau
}{2N}\frac{d^{2}}{dz^{2}}\right)  $ on polynomials $p^{N}$ of degree $N$ by
the terminating power series
\begin{equation}
\exp\left(  \frac{\tau}{2N}\frac{d^{2}}{dz^{2}}\right)  p^{N}(z)=\sum
_{k=0}^{\infty}\frac{1}{k!}\left(  \frac{\tau}{2N}\right)  ^{k}\left(
\frac{d^{2}}{dz^{2}}\right)  ^{k}p^{N}(z),\quad\tau\in\mathbb{C}.
\label{heatDef}%
\end{equation}
We note the operator $d^{2}/dz^{2}$ appearing in (\ref{heatDef}) is not the
Laplacian on the plane (which would be, up to a constant, $\partial
^{2}/\partial z\partial\bar{z}$) but rather the holomorphic extension of the
Laplacian $d^{2}/dx^{2}$ on the line. Replacing $d^{2}/dz^{2}$ in
(\ref{heatDef}) by $\partial^{2}/\partial z\partial\bar{z}$ would not give an
interesting result, since the holomorphic function $p^{N}(z)$ is annihilated
by $\partial/\partial\bar{z}$ and we would simply get back $p^{N}(z).$

If, for example, we take $\tau=-t,$ with $t>0,$ we obtain the backward heat
flow (with a polynomial initial condition). This case is of particular
interest because it preserves the class of polynomials with all real roots, by
the P\'{o}lya--Benz theorem \cite[Theorem 1.2]{Aleman}. The asymptotic
behavior of real roots under the backward heat flow has a natural free
probability interpretation: If the root distribution of a sequence $p^{N}$ of
polynomials approaches a compactly supported probability measure $\mu,$ the
root distribution of $\exp\left(  -\frac{t}{2N}\frac{d^{2}}{dz^{2}}\right)
p^{N}$ approaches the free additive convolution $\mu\boxplus\mathsf{sc}_{t},$
where $\mathsf{sc}_{t}$ is the semicircular measure with variance $t.$ (Free
convolution with the semicircular distribution was investigated by Biane
\cite{BianeHeat} as a free version of the heat equation.) See results of
Kabluchko \cite{Kabluchko} and Voit and Woerner \cite{VW}.  

In this paper, we will investigate the behavior of \textit{complex} roots of
polynomials under the heat flow. One may think, for example, of starting with
a polynomial with real roots and applying the \textit{forward} heat flow
($\tau=t>0$), in which case, the roots will not remain real. (Terry Tao's blog
post \cite{Tao1}, for example, shows that the roots cannot remain real for all
time, except in the case of polynomial of degree 1.) The behavior of complex
roots under the heat flow is much more complicated than for real roots under
the backward heat flow. We are, nevertheless, able to formulate some precise
conjectures about the evolution of complex roots and we find, again, a
connection to random matrix theory and free probability. In Section
\ref{multmult.sec}, we will also study a \textquotedblleft
multiplicative\textquotedblright\ version of the heat flow with similar
conjectures and connections to random matrix theory.

We could conceivably restrict attention to the forward heat flow, taking
$\tau$ to be real and positive. But once the roots are allowed to be complex,
there is no meaningful differences in behavior among the forward heat flow,
the backward heat flow, and the heat flow with a complex time parameter.
Indeed, it makes many of the computations simpler to work with a general
complex time parameter.

The study of roots of polynomials under the heat flow is closely connected to
the study of roots of polynomials under repeated differentiation, as in (to
mention just a very few of the most directly relevant examples) work of
O'Rourke and Steinerberger \cite{OS}, Feng and Yao \cite{FY}, Hoskins and
Kabluchko \cite{HK}, and our own paper with Jalowy and Kabluchko \cite{HHJK3}.

\subsection{Motivation from random matrix theory\label{motivation.sec}}

In this subsection, we motivate the study of roots of polynomials under the
heat flow by connecting it to random matrix theory. In recent years,
connections have been developed between constructions in free probability or
random matrix theory (on the one hand) and operations on polynomials (on the
other hand). The first prominent example of such a connection is the work of
Marcus, Spielman, and Srivastava \cite{MSS} on finite free probability. In the
additive case, the authors of \cite{MSS} take an old construction of Walsh
\cite{Walsh} taking two polynomials $p$ and $q$ of degree $d$ and associating
to it a new polynomial $p\boxplus_{d}q.$ They give it a new interpretation as
the polynomial counterpart of the free additive convolution in free
probability theory. Indeed, results of Arizmendi and Perales \cite[Corollary
5.5]{AP} (see also \cite[Section 4]{Marcus}) show that as $d\rightarrow
\infty,$ the finite free convolution converges in an appropriate sense to the
ordinary free additive convolution.

Another connection between free probability and polynomials is the
relationship, discussed in Section \ref{introFlow.sec}, between the backward
heat flow on real-rooted polynomials and the free convolution with the
semicircular distribution.

Still another such connection comes from repeated differentiation of
polynomials with real roots, in which we take $\lfloor Nt\rfloor$ derivatives
of a polynomial of degree $N,$ with $0\leq t<1.$ Work of Hoskins and Kabluchko
\cite{HK} and Arizmendi, Garza-Vargas, and Perales \cite{AGP} establish the
following result: If a sequence of real-rooted polynomials $p^{N}$ of degree
$N$ have root distribution converging to $\mu,$ then the $\lfloor Nt\rfloor
$-th derivative of $p^{N}$ has root distribution converging to a rescaling of
$\mu_{t}=\mu^{1/(1-t)},$ where $\mu^{\alpha},$ $\alpha\geq1,$ denotes the
fractional free convolution introduced by Bercovici--Voiculescu \cite{BV} and
Nica--Speicher \cite{NicaSpeicher}.

Meanwhile, Campbell, O'Rourke, and Renfrew \cite{COR} have introduced a
version of the preceding result for random polynomials with independent
coefficients; such polynomials have \textit{complex} roots with an
asymptotically radial distribution. Their result also involves a version of
fractional convolution, but for $R$-diagonal operators.

In this paper, we will consider a connection between the heat flow on
polynomials with complex roots and something called the \textquotedblleft
model deformation phenomenon\textquotedblright\ in random matrix theory and
free probability. We now introduce the simplest example of this phenomenon.
Consider a family of \textquotedblleft elliptic\textquotedblright\ random
matrix models, having the form%
\begin{equation}
Z_{t}^{N}=\sqrt{1-\frac{t}{2}}~X^{N}+i\sqrt{\frac{t}{2}}~Y^{N},\quad0\leq
t\leq2,\label{ZtIntro}%
\end{equation}
where $X^{N}$ and $Y^{N}$ are independent GUE matrices. Then $Z_{t}$ is a GUE
matrix at $t=0,$ a Ginibre matrix at $t=1,$ and a GUE matrix multiplied by $i$
at $t=2.$ For $0<t<2,$ the limiting eigenvalue distribution of $Z_{t}^{N}$ is
uniform on an ellipse centered at the origin with semi-axes $2-t$ and $t$. At
$t=0$ or $t=2,$ the limiting eigenvalue distribution collapses onto the real
or imaginary axis, respectively. If $\sigma_{t}$ denote the limiting
eigenvalue distribution of $Z_{t},$ it is easily seen that $\sigma_{t}$ is the
push-forward of $\sigma_{1}$ (uniform on the unit disk) under the map
\[
T_{t}(z)=z-(t-1)\bar{z}.
\]

We now generalize the elliptic model by adding an independent Hermitian
matrix:%
\begin{equation}
A_{t}^{N}=X_{0}^{N}+Z_{t}^{N},\label{atmodel}%
\end{equation}
where $X_{0}^{N}$ is Hermitian and independent of $Z_{t}^{N}.$ (A more general
version of this model is considered in Section \ref{additiveConjecture.sec}.)
The distribution of $X_{0}^{N}$ can be quite general; we assume only the
eigenvalue distribution of $X_{0}^{N}$ converges almost surely to a compactly
supported probability measure on $\mathbb{R}.$ If $\mu_{t}$ denotes the
limiting eigenvalue distribution of $A_{t}^{N},$ results of \cite{HHmult}
(adapted from the \textquotedblleft multiplicative\textquotedblright\ case to
the simpler \textquotedblleft additive\textquotedblright\ case) give the
following results.

\begin{itemize}
\item For all $t$ with $-1<t<1,$ the log potential $S(t,z)$ of $\mu_{t}$
satisfies the PDE
\begin{equation}
\frac{\partial S}{\partial t}=\frac{1}{4}\left(  \left(  \frac{\partial
S}{\partial x}\right)  ^{2}-\left(  \frac{\partial S}{\partial y}\right)
^{2}\right)  , \label{introPDES}%
\end{equation}
except possibly on the boundary of the support of $\mu_{t}.$

\item For all $t$ with $-1<t<1,$ the measure $\mu_{t}$ is the push-forward of
$\mu_{0}$ under the map $T_{t}$ given by%
\[
T_{t}(z)=z-tm_{0}(z),
\]
where $m_{0}$ is the Cauchy transform of the measure $\mu_{0}$:%
\[
m_{0}(z)=\int_{\mathbb{C}}\frac{1}{z-w}~d\mu_{0}(w).
\]

\end{itemize}

The two points above are related as follows:\ The curves $t\mapsto T_{t}(z)$
are precisely the characteristic curves of the PDE\ (\ref{introPDES}). We
refer to the second result as \textquotedblleft model
deformation,\textquotedblright\ meaning that as we deform the random matrix
model by changing $t,$ the limiting eigenvalue distribution deforms by
push-forward under a certain canonical map. As noted above, the paper
\cite{HHmult} gives a multiplicative version of the preceding results.
Furthermore, the paper \cite{Zhong2} extends the second point to the case in
which $X_{0}^{N}$ is no longer required to be Hermitian. Finally, a version of
model deformation is obtained in \cite[Section 6]{HHJK3} with connection to
repeated applications of differential operators to random polynomials.

\begin{idea}
\label{heatDeform.idea}The heat flow is the polynomial counterpart of the
model deformation phenomenon.
\end{idea}

This idea is explained in detail in Section \ref{RMT.sec}, but briefly it
means that applying the heat flow for time $(t-t_{0})$ to the characteristic
polynomial of the random matrix $A_{t_{0}}$ in (\ref{atmodel}) should give a
new polynomial whose roots resemble the eigenvalues of $A_{t}.$ Thus, roughly
speaking, applying the heat operator has the same effect as changing the value
of $t$ in the model. This idea gives one reason to study the heat flow on
polynomials---and it also gives considerable insight into how the roots of
polynomials should evolve under the heat flow.

The basis for Idea \ref{heatDeform.idea} is a \textquotedblleft deformation
theorem\textquotedblright\ for the second moments of the characteristic
polynomials of the matrices $A_{t}^{N}.$ It states that varying $t$ in the
random matrix model has the same effect---at the level of the second
moments---as applying the heat flow. See Theorem \ref{deformationAdditive.thm}
in Section \ref{secMom.sec}.

\subsection{The main conjectures}

\begin{definition}
If $p^{N}$ is a polynomial of degree $N,$ we define the \textbf{empirical root
distribution} of $p^{N}$ as the probability measure on $\mathbb{C}$ given by%
\[
\frac{1}{N}\sum_{j=1}^{N}\delta_{z_{j}},
\]
where $z_{1},\ldots,z_{N}$ are the roots of $p^{N}$, listed with their
multiplicity. If $\{p^{N}\}_{N=1}^{\infty}$ is a sequence of polynomials of
degree $N,$ we say that a measure $\mu$ is the \textbf{limiting root
distribution} of the sequence if the empirical root distribution of $p^{N}$
converges weakly to $\mu$ as $N\rightarrow\infty.$
\end{definition}

\begin{definition}
\label{PotentialCauchy.def}If $\mu$ is a compactly supported probability
measure on $\mathbb{C},$ we define the \textbf{log potential}
$S:\mathbb{C\rightarrow R}$ and the \textbf{Cauchy transform} $m:\mathbb{C}%
\rightarrow\mathbb{C}$ of $\mu$ by%
\begin{align*}
S(z)  &  =\int_{\mathbb{C}}\log(\left\vert z-w\right\vert ^{2})~d\mu(w);\\
m(z)  &  =\int_{\mathbb{C}}\frac{1}{z-w}~d\mu(w).
\end{align*}
These functions are defined Lebesgue almost everywhere for any such $\mu$ and
are defined everywhere if $\mu$ has a bounded density with respect to the
Lebesgue measure on the plane.
\end{definition}

We begin by considering the ordinary heat flow in (\ref{heatDef}). We refer to
this flow as the \textquotedblleft additive\textquotedblright\ (as opposed to
\textquotedblleft multiplicative\textquotedblright) case, because it is
connected (Section \ref{RMT.sec}) to Gaussian random matrix models, which can
be constructed as sums (as opposed to products) of large numbers of
independent random matrices.

We now present a general conjecture for the behavior of high-degree
polynomials under the heat flow.

\begin{conjecture}
\label{introConj.conj}Suppose $p_{0}^{N}$ is a sequence of polynomials of
degree $N$ having a compactly supported limiting root distribution $\mu_{0}$
whose Cauchy transform is globally Lipschitz. Then for all $\tau\in
\mathbb{C},$ the empirical root measure of
\begin{equation}
p_{\tau}^{N}(z):=\exp\left(  \frac{\tau}{2N}\frac{d^{2}}{dz^{2}}\right)
p^{N}(z) \label{pTauDef}%
\end{equation}
also converges weakly to a compactly supported probability measure $\mu_{\tau
}.$ Furthermore, there exists a constant $C>0$ such that the following results hold.

\begin{enumerate}
\item \label{pde.point}The log potential $S(z,\tau)$ of $\mu_{\tau}$ is a
$C^{1,1}$ solution of the PDE%
\begin{equation}
\frac{\partial S}{\partial\tau}=\frac{1}{2}\left(  \frac{\partial S}{\partial
z}\right)  ^{2} \label{addPDEintro}%
\end{equation}
in the set $\{z\in\mathbb{C},~\left\vert \tau\right\vert <C\}.$ Here
$\partial/\partial\tau$ and $\partial/\partial z$ denote the Cauchy--Riemann
operators, given by%
\[
\frac{\partial}{\partial\tau}=\frac{1}{2}\left(  \frac{\partial}{\partial
\tau_{1}}-i\frac{\partial}{\partial\tau_{2}}\right)  ;\quad\frac{\partial
}{\partial z}=\frac{1}{2}\left(  \frac{\partial}{\partial x}-i\frac{\partial
}{\partial y}\right)  ,
\]
where $\tau=\tau_{1}+i\tau_{2}$ and $z=x+iy.$

\item \label{push.point}Let $m_{0}$ be the Cauchy transform of $\mu_{0}$ and
define a \textbf{transport map} $T_{\tau}:\mathbb{C}\rightarrow\mathbb{C}$ by
\begin{equation}
T_{\tau}(z)=z-\tau m_{0}(z). \label{TofTau}%
\end{equation}
Then for all $\tau$ with $\left\vert \tau\right\vert <C$, the measure
$\mu_{\tau}$ is the push-forward of $\mu_{0}$ by $T_{\tau}$:%
\begin{equation}
\mu_{\tau}=(T_{\tau})_{\ast}(\mu_{0}), \label{PushIntro}%
\end{equation}
where $(f)_{\ast}(\mu)$ denotes the push-forward of a measure $\mu$ by a map
$f.$
\end{enumerate}
\end{conjecture}

In Point \ref{pde.point}, the condition $C^{1,1}$ means, more precisely, that
the partial derivatives of $S$ with respect to the real and imaginary parts of
$z$ are globally Lipschitz.

Suppose the roots of $p_{\tau}^{N}$ are distinct for $\tau$ in some domain
$U\subset\mathbb{C},$ so that we can parameterize the roots holomorphically in
$\tau\in U$ as $\{z_{j}(\tau)\}_{j=1}^{N}.$ Then in light of (\ref{TofTau})
and (\ref{PushIntro}), it is natural to expect that the roots will evolve as%
\begin{equation}
z_{j}(\tau)\approx z_{j}(0)-\tau m_{s,\tau_{0}}(z_{j}(0)).
\label{approxMotion}%
\end{equation}
We may interpret (\ref{approxMotion}) to mean that, with high probability when $N$ is large, 
most of trajectories $z_j(\tau)$ will be uniformly close to the right-hand side of (\ref{approxMotion}), 
for sufficiently small $\tau$. 
Although such a result does not follow from (\ref{PushIntro}), the arguments
in Section \ref{dynamical.sec} will support this behavior. A rigorous result in this direction was established for the zeros of the heat-evolved Gaussian analytic function in Theorem 1.3 of our paper with Jalowy and Kabluchko \cite{HHJK1}.

The simplest case of Conjecture \ref{introConj.conj} is the one in which
$\mu_{0}$ is uniform on the unit disk, in which case $m_{0}(z)=\bar{z}$ for
all $z$ in the disk. Then $T_{\tau}(z)=z-\tau\bar{z}$ and for real $\tau$
between $-1$ and $1,$ the measure $\mu_{\tau}$ in (\ref{PushIntro}) is uniform
on an ellipse with semi-axes $1-\tau$ and $1+\tau.$ See Figure
\ref{uniformlattice.fig} and see also Slide 3 in the supplemental document for
an animation.%

%TCIMACRO{\FRAME{ftbpFU}{4.8282in}{1.9441in}{0pt}{\Qcb{Heat evolution of a
%polynomial whose roots form a uniform lattice inside the unit disk. Shown for
%$\tau=-1/2.$}}{\Qlb{uniformlattice.fig}}{uniformlattice.pdf}%
%width 4.8282in;  height 1.9441in;  depth 0pt;  original-width 8.0004in;
%original-height 3.1946in;  cropleft "0";  croptop "1";  cropright "1";
%cropbottom "0";  filename 'UniformLattice.pdf';file-properties "XNPEU";}} }%
%BeginExpansion
\begin{figure}[ptb]%
\centering
\includegraphics[
%%natheight=3.194600in,
%%natwidth=8.000400in,
height=1.9441in,
width=4.8282in
]%
{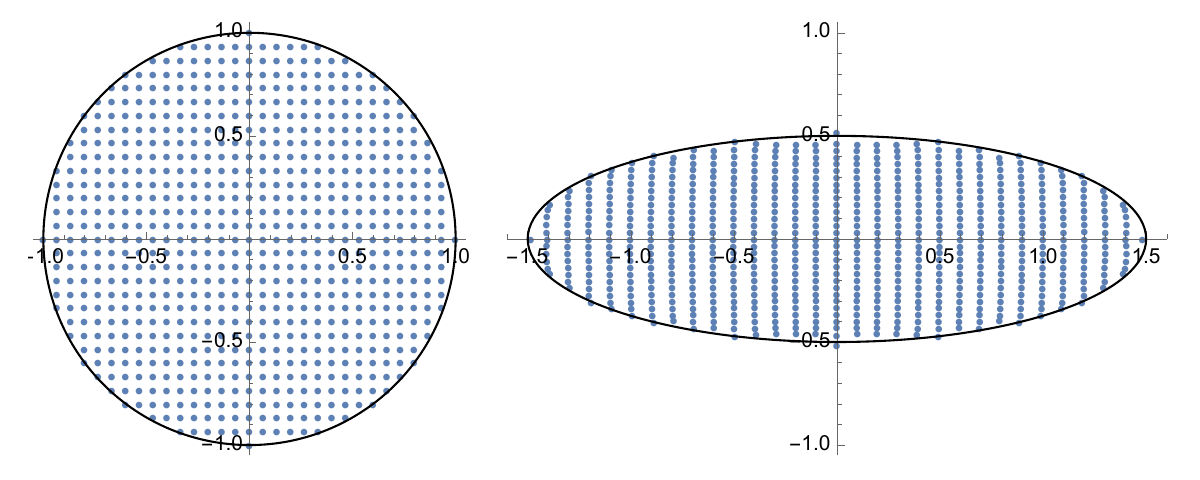}%
\caption{Heat evolution of a polynomial whose roots form a uniform lattice
inside the unit disk. Shown for $\tau=-1/2.$}%
\label{uniformlattice.fig}%
\end{figure}
%EndExpansion

\begin{remark}
If we take $\tau$ to equal a real number $t$ in (\ref{addPDEintro}), take the
real part of both sides, and then multiply by 2, we obtain the following
real-variables PDE:%
\[
\frac{\partial S}{\partial t}=\frac{1}{4}\left(  \left(  \frac{\partial
S}{\partial x}\right)  ^{2}-\left(  \frac{\partial S}{\partial y}\right)
^{2}\right)  ,\quad t\in\mathbb{R}.
\]
Here $\partial/\partial t$ is the ordinary partial derivative in the real
variable $t.$
\end{remark}

The assumption that $m_{0}(z)$ is globally Lipschitz guarantees that the map
$T_{\tau}$ is a homeomorphism of $\mathbb{C}$ to $\mathbb{C}$, when
$\left\vert \tau\right\vert $ is small enough; see Lemma \ref{lem.Ttau.prop}.
This assumption will be satisfied if $\mu_{0}$ is absolutely continuous with
respect to Lebesgue measure on the plane, with a sufficiently smooth density.
But if $\mu_{0}$ is supported on the real line, $m_{0}$ will not even be continuous.

As explained in Sections \ref{brief.sec} and \ref{details.sec}, if we have
enough regularity in the solution to (\ref{addPDEintro}), we can hope to
deduce Point \ref{push.point} of Conjecture \ref{pde.point} from Point
\ref{pde.point}. To do this, we will take the Laplacian of both sides of
(\ref{addPDEintro}) to obtain a continuity equation satisfied by the density
of the measure $\mu_{\tau},$ from which a push-forward theorem will hold,
under suitable technical assumptions. But in Section \ref{dynamical.sec}, we
will give a direct argument for Point \ref{push.point} of the conjecture,
which does not rely on Point \ref{pde.point}.

\begin{remark}
The polynomials in Conjecture \ref{introConj.conj} can be deterministic or
random. In the random case, one might imagine proving the result only
\textquotedblleft in probability.\textquotedblright\ That is, one might prove
that both the hypothesis (convergence of the empirical root measure of
$p_{0}^{N}$) and the conclusions (Points \ref{pde.point} and \ref{push.point})
hold in the sense of weak convergence in probability. 
\end{remark}

\textquotedblleft Weak convergence in probability\textquotedblright\ means
that with with high probability, the approximating measures are close to the
limit measure in the weak topology.

\begin{remark}
\label{hhjk.remark}In a separate work with Jalowy and Kabluchko \cite{HHJK2},
we prove an \textquotedblleft in probability\textquotedblright\ version of
Conjecture \ref{introConj.conj} for random polynomials with independent
coefficients. One needs mild regularity assumptions on the coefficients to
guarantee that the constant $C$ (denoted $t_{\mathrm{sing}}$ in \cite{HHJK2})
is positive. See Theorem 5.2 (PDE) and Theorem 3.3 (push-forward theorem) in
\cite{HHJK2}.

In another work with the same authors \cite{HHJK1}, we establish rigorous
results in \cite{HHJK1} for the zeros of the Gaussian analytic function under
the heat flow. See also \cite{HHJK3} for similar results on the zeros of
random polynomials under repeated differentiation.
\end{remark}

The work \cite{HHJK2} builds on work of Kabluchko and Zaporozhets \cite{KZ}.
The paper \cite{KZ} studies polynomials formed as sums of ordinary monomials
$z^{j}$ multiplied by random coefficients that are independent (but typically
not identically distributed). The paper \cite{HHJK2} applies the heat flow to
a Kabluchko--Zaporozhets polynomial---which has the effect of replacing the
monomial $z^{j}$ by the (properly scaled) $j$th Hermite polynomial. The heat
flow therefore leads to a major new class of random polynomials that can be
analyzed by methods similar to those in \cite{KZ}. This observation then opens
the door to the study of even more general random polynomials by similar
methods (work in progress by the present authors with Jalowy and Kabluchko).

We present three distinct lines of reasoning for Conjecture
\ref{introConj.conj}.

\begin{itemize}
\item An argument from the perspective of \textbf{random matrix theory}
(Section \ref{RMT.sec}). This line of reasoning is specific to the case in
which $p_{0}^{N}$ is the characteristic polynomial of a certain class of
random matrix models, and we would hope for an almost sure version of the conjecture.

\item An argument from the perspective of \textbf{dynamical systems} (Section
\ref{dynamical.sec}). We will argue that for large $N,$ the roots move in
approximately straight lines with constant velocity given by the negative of
initial Cauchy transform. This argument supports Point \ref{push.point} of the conjecture.

\item An argument from the perspective of \textbf{partial differential
equations} (Section \ref{pde.sec}). We will derive an exact PDE satisfied by
the log potential of the empirical root measure of a polynomial evolving by
the heat equation. This PDE\ formally converges to the PDE (\ref{addPDEintro})
in the limit $N\rightarrow\infty.$ We then show how the
PDE\ (\ref{addPDEintro}) implies a continuity equation for the density of the
measure. With enough regularity, we can then show that Point \ref{pde.point}
of Conjecture \ref{introConj.conj} implies Point \ref{push.point}.
\end{itemize}

We now present the \textquotedblleft multiplicative\textquotedblright\ version
of Conjecture \ref{introConj.conj}. The heat flow in (\ref{heatDefMult})
arises naturally from certain multiplicative random matrix models. See Theorem
\ref{deformationMult.thm} and Lemma \ref{secondIDmult.lem}. This flow is
almost the same as the one discussed by Terry Tao in \cite{Tao2}.

\begin{conjecture}
\label{introConjMult.conj}A \textquotedblleft multiplicative\textquotedblright%
\ version of Conjecture \ref{introConj.conj} also holds, with the heat flow in
(\ref{heatDef}) replaced by the operator%
\begin{equation}
\exp\left\{  -\frac{(\tau-\tau_{0})}{2N}\left(  z^{2}\frac{\partial^{2}%
}{\partial z^{2}}-(N-2)z\frac{\partial}{\partial z}-N\right)  \right\}  ,
\label{heatDefMult}%
\end{equation}
the PDE (\ref{addPDEintro}) replaced by
\[
\frac{\partial S}{\partial\tau}=-\frac{1}{2}\left(  z^{2}\left(
\frac{\partial S}{\partial z}\right)  ^{2}-z\frac{\partial S}{\partial
z}\right)
\]
and the transport map $T_{\tau}$ in (\ref{TofTau}) replaced by
\[
\tilde{T}_{\tau}(z)=z\exp\left\{  \tau zm_{0}(z)-\frac{1}{2}\right\}  ,
\]
where $m_{0}$ is the Cauchy transform of the limiting root distribution of the
initial polynomials.
\end{conjecture}

Since the operator in the exponent in (\ref{heatDefMult}) maps $z^{j}$ to a
multiple of $z^{j}$ for all $j,$ the exponentiated operator will also have
this property. Conjecture \ref{introConjMult.conj} will be supported by
appropriate versions of the three lines of reasoning described above
supporting Conjecture \ref{introConj.conj}.

\subsection{Rigorous results}

None of the perspectives discussed above leads to a rigorous argument for
Conjecture \ref{introConj.conj} or Conjecture \ref{introConjMult.conj}. We
nevertheless obtain some rigorous results, as follows. First, we establish in
Theorems \ref{deformationAdditive.thm} and \ref{deformationMult.thm} a deformation
theorems for the second moment of the characteristic polynomial of certain
random matrix models, in both the \textquotedblleft additive\textquotedblright%
\ and \textquotedblleft multiplicative\textquotedblright\ settings. This
result provides a conceptual link between (on the one hand) the deformation of
certain random matrix models with respect to a parameter and (on the other
hand) the deformation of the associated characteristic polynomials by an
appropriate heat flow. We expect that a version of Theorem \ref{deformationAdditive.thm} will
hold for other families of random matrices, leading to
further connections between random matrices and differential flows on polynomials.
 Second, we show in Section \ref{moments.sec} that our
two main conjectures hold at the level of the holomorphic moments. Although
these moments do not uniquely determine the limiting root distributions, the
results of Section \ref{moments.sec} provide substantial confirmation of our
conjectures. We also remind the reader of rigorous results obtained for random
polynomials, in a separate paper with Jalowy and Kabluchko. See Remark
\ref{hhjk.remark}.

\section{A random matrix perspective\label{RMT.sec}}

In this section, we present a conjecture related to the \textquotedblleft
model deformation phenomenon\textquotedblright\ discussed in Section
\ref{motivation.sec}. A multiplicative version of this conjecture is presented
in Section \ref{multiplicativeConj.sec}.

\subsection{The heat flow conjecture for random
matrices\label{additiveConjecture.sec}}

We consider now a generalization of the random matrix model denoted $A_{t}$ in
(\ref{atmodel}) in Section \ref{motivation.sec}. Let $X^{N}$ and $Y^{N}$ be
independent $N\times N$ GUE matrices and consider a matrix of the form%
\begin{equation}
Z^{N}=e^{i\theta}(aX^{N}+ibY^{N}),\label{ZNdef}%
\end{equation}
where $a$ and $b$ are real numbers, assumed not both zero. We call such a
matrix a (rotated) \textbf{elliptic random matrix} (with the parameter
$\theta$ giving the rotation). It is convenient to parameterize such matrices
by a real, positive variance parameter $s$ and a complex covariance parameter
$\tau,$ given by%
\begin{align}
s &  =\mathbb{E}\left\{  \frac{1}{N}\operatorname{Trace}((Z^{N})^{\ast}%
Z^{N})\right\}  \nonumber\\
\tau &  =\mathbb{E}\left\{  \frac{1}{N}\operatorname{Trace}((Z^{N})^{\ast
}Z^{N})\right\}  -\mathbb{E}\left\{  \frac{1}{N}\operatorname{Trace}%
((Z^{N})^{2})\right\}  .\label{sAndTau}%
\end{align}
This parameterization is motivated by \cite{DHKcomplex}, where (in the case
$K_{\mathbb{C}}=\mathbb{C}^{N^{2}}$) the law of $Z^{N}$ is denoted as
$\mu_{s,\tau}$ and where the parameter $\tau$ plays the role of the complex
time variable in the Segal--Bargmann transform.

The parameters $s$ and $\tau$ completely determine the distribution of the
matrix $Z^{N}$; see \cite[Section 2.1]{HHmult} for details. Using the
Cauchy--Schwarz inequality, we can verify that
\begin{equation}
\left\vert \tau-s\right\vert \leq s.\label{tauIneq}%
\end{equation}
We label such a matrix as $Z_{s,\tau}^{N}$ and we then consider a random
matrix of the form%
\[
X_{0}^{N}+Z_{s,\tau}^{N},
\]
where $X_{0}^{N}$ is independent of $Z_{s,\tau}^{N}.$ We assume that all of
the expected $\ast$-moments of $X_{0}^{N}$---that is, expectation values of
normalized traces of words in $X_{0}^{N}$ and its adjoint---are well defined
and finite. The model (\ref{atmodel}) considered in Section
\ref{motivation.sec} corresponds to the case in which $s=1,$ $\tau$ is real,
and $X_{0}^{N}$ is Hermitian. 

Suppose now that $X_{0}^{N}$ converges almost surely in $\ast$-distribution to
an element $x_{0}$ in a tracial von Neumann algebra. Define%
\begin{equation}
z_{s,\tau}=e^{i\theta}(ax+iby) \label{zstauFree}%
\end{equation}
where $x$ and $y$ are freely independent semicircular elements and the
parameters $s$ and $\tau$ are define by analogy to (\ref{sAndTau}). When
$\left\vert \tau-s\right\vert <s,$ we may apply a result of \'{S}niady
\cite{Sniady} to conclude that the eigenvalue distribution of $X_{0}%
^{N}+X_{s,\tau}^{N}$ converges almost surely to the Brown measure of
$x_{0}+z_{s,\tau},$ where $x_{0}$ and $z_{s,\tau}$ are taken to be freely
independent. (When $\left\vert \tau-s\right\vert <s,$ the element $Z^{N}$ in
(\ref{ZNdef}) is equal in distribution to the sum of a multiple of a Ginibre
matrix and an independent elliptic matrix, so that Theorem 6 in \cite{Sniady} applies.)

\begin{conjecture}
[Arbitrary plus elliptic model]\label{ellipticPlus.conj}Fix $s>0$ and two
complex numbers $\tau_{0}$ and $\tau$ with $\left\vert \tau_{0}-s\right\vert
<s$ and $\left\vert \tau-s\right\vert \leq s.$ Let $\mu_{s,\tau_{0}}$ be the
Brown measure of $x_{0}+z_{s,\tau_{0}},$ where $z_{s,\tau_{0}}$ is defined as
in (\ref{zstauFree}) and where $x_{0}$ is freely independent of $z_{s,\tau
_{0}},$ and let $m_{s,\tau_{0}}$ be its Cauchy transform. Then $m_{s,\tau_{0}%
}$ is Lipschitz continuous and the Brown measure $\mu_{s,\tau}$ of
$x_{0}+z_{s,\tau}$ can be computed by the model deformation formula%
\begin{equation}
\mu_{s,\tau}=(\Phi_{s,\tau_{0},\tau})_{\ast}(\mu_{s,\tau_{0}}),
\label{pushAdd}%
\end{equation}
where the map $\Phi_{s,\tau_{0},\tau}:\mathbb{C}\rightarrow\mathbb{C}$ is
defined as%
\[
\Phi_{s,\tau_{0},\tau}(z)=z-(\tau-\tau_{0})m_{s,\tau_{0}}(z).
\]
Furthermore, as $\tau$ varies over the set $\left\vert \tau-s\right\vert <s$
(with $x$ and $s$ fixed), the log potential of $S_{s,\tau}$ is a $C^{1}$
solution of the PDE%
\begin{equation}
\frac{\partial S_{s,\tau}}{\partial\tau}=\frac{1}{2}\left(  \frac{\partial
S_{s,\tau}}{\partial z}\right)  ^{2}. \label{AddConjPDE}%
\end{equation}

\end{conjecture}

\begin{remark}
Substantial partial results in the direction of Conjecture
\ref{ellipticPlus.conj} have been obtained. First, when $x_{0}$ is
self-adjoint, the \textquotedblleft additive\textquotedblright\ counterpart of
the results of Hall--Ho \cite{HHmult} apply. (The proofs from \cite{HHmult}%
---which addresses the \textquotedblleft multiplicative\textquotedblright%
\ case---apply with minor modifications in the easier additive case.) In this
case, the push-forward result (\ref{pushAdd}) holds and the PDE
(\ref{AddConjPDE}) holds at least away from the boundary of the support of
$\mu_{s,\tau}.$ See Theorem 4.2, Proposition 6.3, Corollary 7.7 and Theorem
8.2 in \cite{HHmult}. Second, for general $x_{0}$ (assumed only to be freely
independent of $z_{s,\tau_{0}}$), the results of Zhong \cite{Zhong2} apply. A
push-forward result holds in this generality, at least for $\tau_{0}=s$
\cite[Theorem C]{Zhong2}. But since Zhong does not use PDE\ methods in
\cite{Zhong2}, he does not address whether the PDE (\ref{AddConjPDE}) holds.
\end{remark}

We emphasize that Conjecture \ref{ellipticPlus.conj} is not the main focus of
this paper. We present it to explain how Conjecture \ref{add1.conj} below
relates to the general proposal given in Conjecture \ref{introConj.conj}. The
idea is that the change from $\tau_{0}$ to $\tau$ in Conjecture
\ref{ellipticPlus.conj} can be carried out by applying the heat flow to the
characteristic polynomial of the random matrix with the first value of $\tau.$

\begin{conjecture}
\label{add1.conj}Fix $s>0$ and complex numbers $\tau_{0}$ and $\tau$ with
$\left\vert \tau_{0}-s\right\vert \leq s$ and $\left\vert \tau-s\right\vert
\leq s.$ Let \thinspace$Z_{s,\tau_{0}}^{N}$ and $Z_{s,\tau}^{N}$ be elliptic
random matrices with parameters defined as in (\ref{sAndTau}). Let $X_{0}^{N}$
be independent of $Z_{s,\tau}$ and $Z_{s,\tau_{0}}$ and assume $X_{0}^{N}$
converges almost surely in $\ast$-distribution to an element $x_{0}$ in a
tracial von Neumann algebra. Let $p_{s,\tau_{0}}$ be the random characteristic
polynomial of $X_{0}^{N}+Z_{s,\tau_{0}}^{N}$ and define a new random
polynomial $q_{s,\tau_{0},\tau}$ by%
\[
q_{s,\tau_{0},\tau}(z)=\exp\left\{  \frac{(\tau-\tau_{0})}{2N}\frac
{\partial^{2}}{\partial z^{2}}\right\}  p_{s,\tau_{0}}(z),\quad z\in
\mathbb{C}.
\]
Then the empirical root measure of $q_{s,\tau_{0},\tau}$ converges weakly
almost surely to the limiting eigenvalue distribution of $X_{0}^{N}+Z_{s,\tau
}^{N}.$
\end{conjecture}

The conjecture says that applying the heat operator for time $\tau-\tau_{0}$
to the polynomial $p_{s,\tau_{0}}$ gives a new polynomial $q_{s,\tau_{0},\tau
}$ whose roots resemble those of $p_{s,\tau}$---namely the eigenvalues of
$X_{0}^{N}+Z_{s,\tau}^{N}.$ Thus, the heat flow effectively changes the value
of $\tau$. We emphasize, however, that the conjecture is \textit{not} claiming
that the joint distribution of the roots of $q_{s,\tau_{0},\tau}$ is the same
as the joint distribution of the eigenvalues of $X_{0}^{N}+Z_{s,\tau}^{N}.$
Rather, the conjecture merely asserts that the \textit{limiting} eigenvalue
distributions of the two collections of points are the same. The points in
Figure \ref{evolvess.fig}, for example, correspond to the case $X_{0}^{N}=0,$
with $s=1,$ $\tau_{0}=0,$ and $\tau=1,$ so that $Z_{s,\tau_{0}}$ is GUE and
$Z_{s,\tau}$ is Ginibre. Although the points in the figure approximate the
uniform distribution on the unit disk, they are clearly distinguishable to the
eye from the eigenvalues of a Ginibre matrix.

\begin{remark}
If Conjectures \ref{ellipticPlus.conj} and \ref{add1.conj} hold, then
Conjecture \ref{introConj.conj} also holds---that is, the expected PDE\ and
push-forward results will hold for the limiting root distribution of
$q_{s,\tau_{0},\tau}$. But the argument we will present for Conjecture
\ref{add1.conj} in Section \ref{secMom.sec} is independent of the validity of
Conjecture \ref{ellipticPlus.conj}.
\end{remark}

\subsection{Examples}%

%TCIMACRO{\FRAME{ftbpFU}{4.0283in}{2.0141in}{0pt}{\Qcb{Eigenvalues of a Ginibre
%matrix (top), eigenvalues of an elliptic matrix (bottom left), and roots of
%the heat-evolution of the characteristic polynonial of the Ginibre matrix
%(bottom right). }}{\Qlb{circtoellip.fig}}{circtoellip.pdf}%
%width 4.0283in;  height 2.0141in;  depth 0pt;  original-width 8.0004in;
%original-height 5.5002in;  cropleft "0";  croptop "1";  cropright "1";
%cropbottom "0";  filename 'CircToEllip.pdf';file-properties "XNPEU";}} }%
%BeginExpansion
\begin{figure}[ptb]%
\centering
\includegraphics[
%%natheight=5.500200in,
%%natwidth=8.000400in,
height=2.0141in,
width=4.0283in
]%
{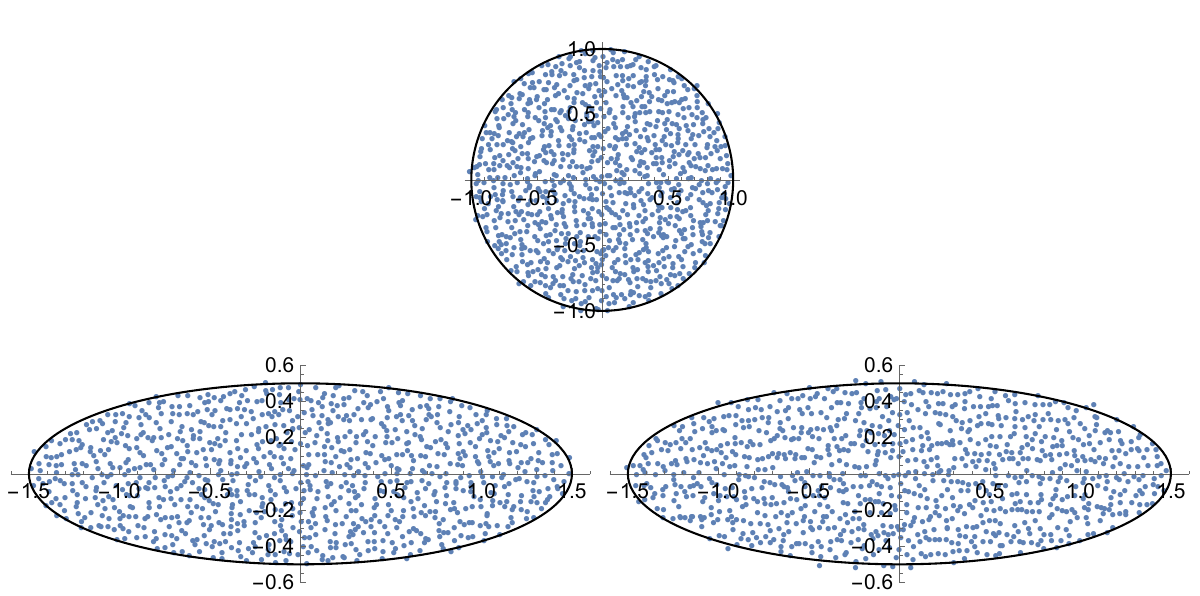}%
\caption{Eigenvalues of a Ginibre matrix (top), eigenvalues of an elliptic
matrix (bottom left), and roots of the heat-evolution of the characteristic
polynonial of the Ginibre matrix (bottom right). }%
\label{circtoellip.fig}%
\end{figure}
%EndExpansion

In this subsection, we present three examples of the general proposal in
Conjecture \ref{add1.conj}. We emphasize that all the examples still have the
status of a conjecture: special cases of the general conjecture. As our first
example, we consider the case in which $X_{0}^{N}$ is zero, and the parameters
in (\ref{sAndTau}) are $s=1$ and $\tau_{0}=1.$ In that case, $X_{0}%
^{N}+Z_{s,\tau_{0}}^{N}=Z_{1,1}^{N}$ is a Ginibre matrix. We then take
$\tau=1-t$, where $t$ is a real number between $-1$ and 1, so that
\[
(\tau-\tau_{0})=-t,\quad-1\leq t\leq1.
\]

\begin{example}
[Circular to elliptic and circular to semicircular]\label{circToSS.example}Let
$Z^{N}$ be an $N\times N$ random matrix chosen from the Ginibre ensemble and
let $p$ be its random characteristic polynomial. Fix a real number $t$ with
$-1\leq t\leq1$ and define a new random polynomial $q_{t}$ by%
\[
q_{t}(z)=\exp\left\{  -\frac{t}{2N}\frac{\partial^{2}}{\partial z^{2}%
}\right\}  p(z),\quad z\in\mathbb{C}.
\]
For $-1<t<1,$ Conjecture \ref{add1.conj} predicts that the empirical root
measure of $q_{t}$ will converge weakly almost surely to the uniform
probability measure on the ellipse centered at the origin with semi-axes $1+t$
and $1-t.$ For $t=1,$ Conjecture \ref{add1.conj} predicts that the empirical
root measure of $q_{1}$ will converge weakly almost surely to the semicircular
probability measure on $[-2,2]\subset\mathbb{R}.$

In this setting, the initial measure $\mu_{s,\tau_{0}}=\mu_{1,1}$ is the
uniform measure on the unit disk, so that $m_{s,\tau_{0}}(z)=\bar{z}$ for all
$z$ in the unit disk. Thus, (\ref{approxMotion}) predicts that, in this case,
the roots $\{z_{j}(t)\}_{j=1}^{N}$ should evolve as
\begin{equation}
z_{j}(t)\approx z_{j}(0)+t\overline{z_{j}(0)},\quad-1\leq t\leq1, \label{zOfT}%
\end{equation}
where $\{z_{j}(0)\}_{j=1}^{N}$ are the eigenvalues of $Z^{N}.$ In particular,
we should have%
\[
z_{j}(1)\approx2\operatorname{Re}[z_{j}(0)].
\]

\end{example}

Figure \ref{circtoellip.fig} shows the similarity between the roots of $q_{t}$
and the eigenvalues of the corresponding elliptic random matrix model. See
also Slide 4 in the supplemental document for an animation. Figure
\ref{traj1.fig}, meanwhile, illustrates the approximate straight-line motion
of the roots predicted in (\ref{zOfT}). See also Slide 5 in the supplemental
document for a dynamical version of Figure \ref{traj1.fig}. We note that a
rigorous version of the first paragraph of Example \ref{circToSS.example} has
been established in \cite[Theorem 2.1]{HHJK3}---but with the characteristic
polynomial of a Ginibre matrix replaced by a Weyl polynomial. (A Weyl
polynomial is a certain random polynomial whose limiting root distribution is
uniform on the unit disk.) In addition, a rigorous result in the direction of
(\ref{zOfT}) has been established for the individual zeros of the Gaussian
analytic function in \cite[Theorem 1.3]{HHJK1}.%

%TCIMACRO{\FRAME{ftbpFU}{3.7775in}{2.0487in}{0pt}{\Qcb{Plots of 100 of the
%curves $z_{j}(t),$ $0\leq t\leq1,$ in Example \ref{circToSS.example}, starting
%from the eigenvalues of a $1,000\times1,000$ Ginibre matrix. Each curve
%changes color from blue to red as $t$ increases.}}{\Qlb{traj1.fig}}%
%{traj1.pdf}{\special{ language "Scientific Word";  type "GRAPHIC";
%width 3.7775in;  height 2.0487in;  depth 0pt;  original-width 5.0004in;
%original-height 2.7086in;  cropleft "0";  croptop "1";  cropright "1";
%cropbottom "0";  filename 'traj1.pdf';file-properties "XNPEU";}} }%
%BeginExpansion
\begin{figure}[ptb]%
\centering
\includegraphics[
%%natheight=2.708600in,
%%natwidth=5.000400in,
height=2.0487in,
width=3.7775in
]%
{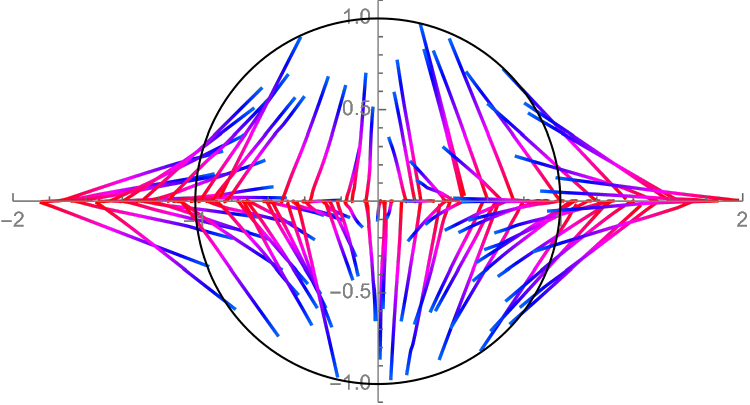}%
\caption{Plots of 100 of the curves $z_{j}(t),$ $0\leq t\leq1,$ in Example
\ref{circToSS.example}, starting from the eigenvalues of a $1,000\times1,000$
Ginibre matrix. Each curve changes color from blue to red as $t$ increases.}%
\label{traj1.fig}%
\end{figure}
%EndExpansion

As our second example of Conjecture \ref{add1.conj}, we consider the case in
which $X_{0}^{N}$ is zero, $s=1,$ and $\tau_{0}=0.$ In that case, $X_{0}%
^{N}+Z_{s,\tau_{0}}^{N}$ is a GUE matrix.

\begin{example}
[Semicircular to elliptic and semicircular to circular]%
\label{ssToCirc.example}Let $X^{N}$ be a GUE matrix and let $p$ be its random
characteristic polynomial. Fix a real number $t$ with $0\leq t\leq2$ and
define a new random polynomial $q_{t}$ by%
\[
q_{t}(z)=\exp\left\{  \frac{t}{2N}\frac{\partial^{2}}{\partial z^{2}}\right\}
p(z),\quad z\in\mathbb{C}.
\]
For $0<t<2,$ Conjecture \ref{add1.conj} predicts that the empirical root
measure of $q_{t}$ will converge weakly almost surely to the uniform
probability measure on the ellipse centered at the origin with semi-axes $2-t$
and $t.$ For $t=2,$ Conjecture \ref{add1.conj} predicts that the empirical
root measure of $q_{t}$ will converge weakly almost surely to the semicircular
distribution on the interval $[-2,2]$ on the imaginary axis.
\end{example}

We note that the eigenvalues of a GUE matrix are distinct with probability 1.
(The joint distribution of the eigenvalues is absolutely continuous on
$\mathbb{R}^{N}$---e.g., \cite[Theorem 2.5.2]{AGZ}--- so the measure of any
hyperplane where two eigenvalues are equal is zero.) Furthermore, inside the
collection of polynomials of degree $N$ with real coefficients, the set of
polynomials with distinct real roots is open. Since the heat evolution is
continuous, we can therefore conclude that the roots of $q_{t}(z)$ remain real
and distinct for sufficiently small positive values of $t.$ We believe,
however, that when $N$ is large, the roots will very quickly begin to collide
and move off the real axis. (This behavior is in the spirit of how the Riemann
xi function---derived from the Riemann zeta function---behaves under the heat
flow; see \cite{Tao3}. ) There is therefore no contradiction in conjecturing
that the $N\rightarrow\infty$ root distribution of $q_{t}$ will be uniform on
an ellipse in the \textit{plane}, even for small positive $t.$%

%TCIMACRO{\FRAME{ftbpFU}{2.8279in}{3.5276in}{0pt}{\Qcb{The characteristic
%polynomial of a GUE matrix and the heat-evolution of the same polynomial at
%$t=0.05$, both multiplied by a Gaussian. Shown for $N=60.$ The number of real
%roots is 60 (top) and 32 (bottom).}}{\Qlb{evolvedpoly.fig}}{evolvedpoly.pdf}%
%width 2.8279in;  height 3.5276in;  depth 0pt;  original-width 8.0004in;
%original-height 9.9998in;  cropleft "0";  croptop "1";  cropright "1";
%cropbottom "0";  filename 'evolvedPoly.pdf';file-properties "XNPEU";}} }%
%BeginExpansion
\begin{figure}[ptb]%
\centering
\includegraphics[
%%natheight=9.999800in,
%%natwidth=8.000400in,
height=3.5276in,
width=2.8279in
]%
{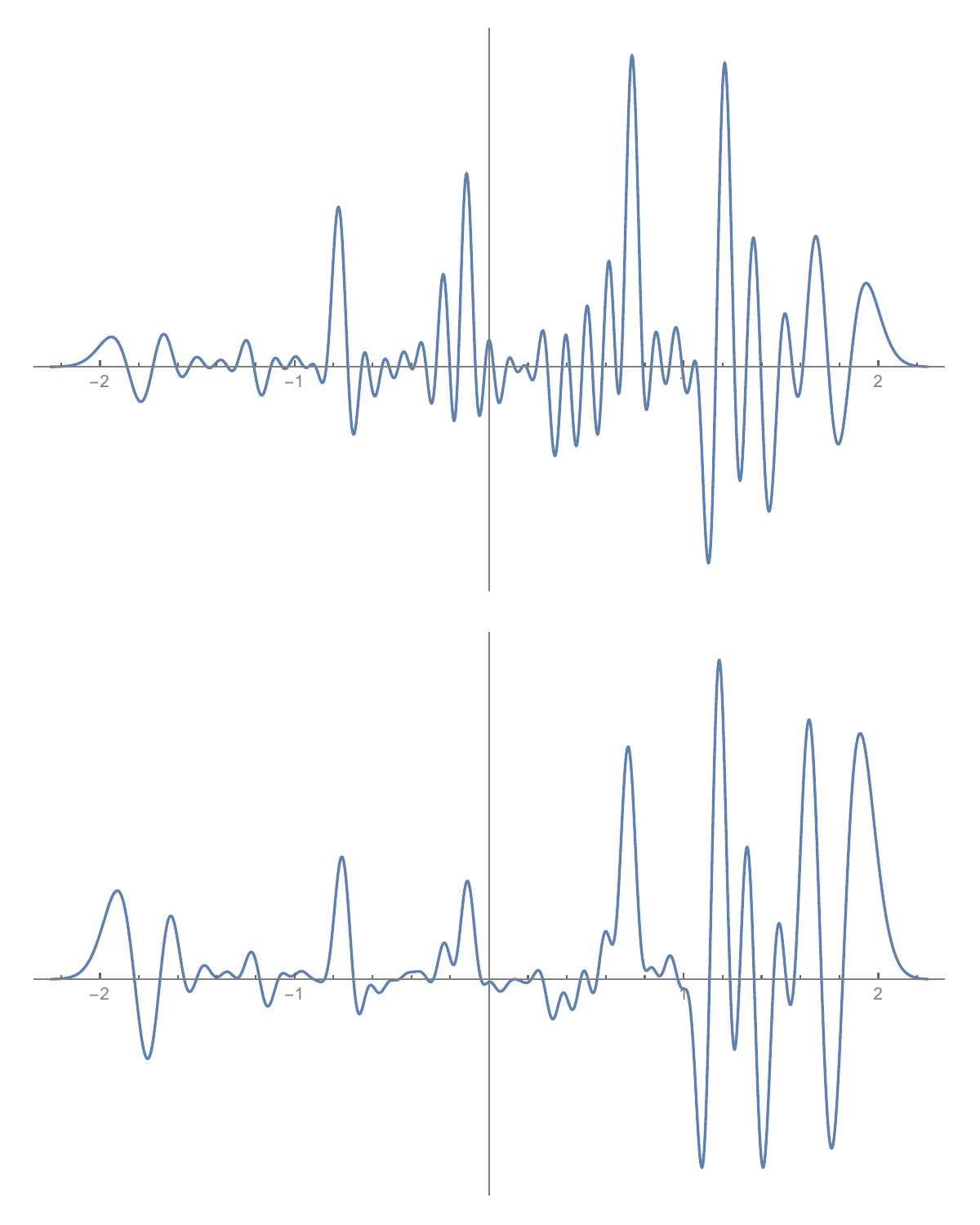}%
\caption{The characteristic polynomial of a GUE matrix and the heat-evolution
of the same polynomial at $t=0.05$, both multiplied by a Gaussian. Shown for
$N=60.$ The number of real roots is 60 (top) and 32 (bottom).}%
\label{evolvedpoly.fig}%
\end{figure}
%EndExpansion

The top part of Figure \ref{evolvedpoly.fig} shows the characteristic
polynomial $p$ of a GUE matrix with $N=60,$ multiplied by a suitable Gaussian
to make the values of a manageable size. (Specifically, it is convenient to
multiply $p(x)$ by $2^{N/2}e^{-Nx^{2}/4},$ which of course does not change the
zeros.) The bottom part of the figure then shows the polynomial $q_{t}$ with
$t=0.05,$ multiplied by the same Gaussian. Already by the time $t=0.05,$ the
number of real roots has dropped from $60$ to 32. Figure \ref{evolvess.fig}
then shows the complex roots of the heat-evolved polynomial at time $t=1.$ See
also Slide 6 in the supplemental document for an animation.%

%TCIMACRO{\FRAME{ftbpFU}{2.1767in}{2.1707in}{0pt}{\Qcb{The points
%$\{z_{j}(1)\}_{j=1}^{N}$ in Example \ref{ssToCirc.example}, with $N=1,000.$
%The points display an obvious banding structure in the vertical direction but
%still approximate a uniform distribution on the unit disk.}}%
%{\Qlb{evolvess.fig}}{evolvess.pdf}{\special{ language "Scientific Word";
%type "GRAPHIC";  maintain-aspect-ratio TRUE;  display "USEDEF";
%valid_file "F";  width 2.1767in;  height 2.1707in;  depth 0pt;
%original-width 5.0004in;  original-height 5.0004in;  cropleft "0";
%filename 'evolveSS.pdf';file-properties "XNPEU";}} }%
%BeginExpansion
\begin{figure}[ptb]%
\centering
\includegraphics[
%%natheight=5.000400in,
%%natwidth=5.000400in,
height=2.1707in,
width=2.1767in
]%
{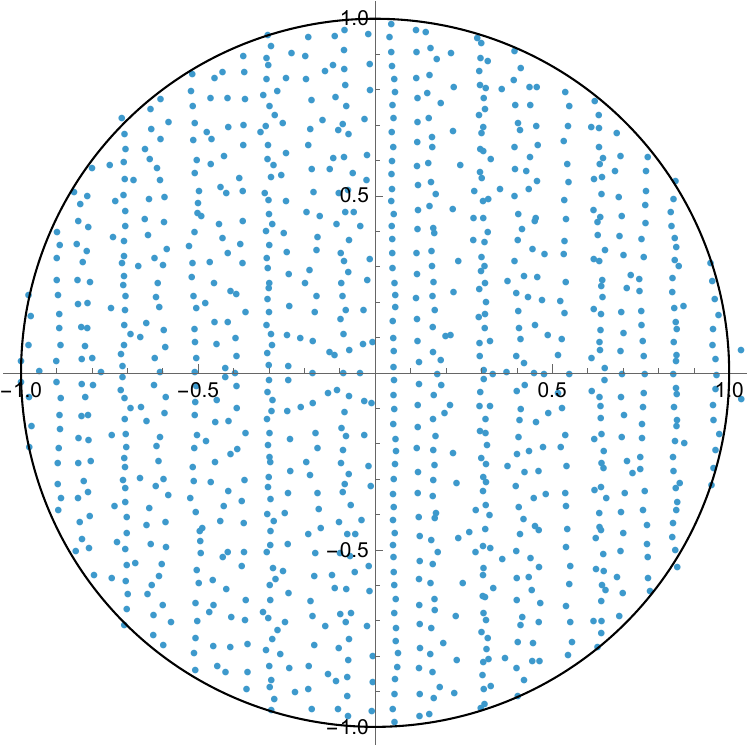}%
\caption{The points $\{z_{j}(1)\}_{j=1}^{N}$ in Example \ref{ssToCirc.example}%
, with $N=1,000.$ The points display an obvious banding structure in the
vertical direction but still approximate a uniform distribution on the unit
disk.}%
\label{evolvess.fig}%
\end{figure}
%EndExpansion

The anticipated trajectories of the roots in Example \ref{ssToCirc.example} is
more complicated than in Example \ref{circToSS.example}. After all, we are
effectively trying to run time backward from $t=1$ in the map (\ref{zOfT}),
even though the $t=1$ map $z\mapsto2\operatorname{Re}z$ is not invertible. To
put it a different way, Example \ref{ssToCirc.example} asserts that we can
deform a \textit{one}-dimensional distribution of points along the real axis
into a \textit{two}-dimensional uniform distribution on an ellipse or disk.
This deformation cannot be achieved by applying a smooth map of the sort we
have on the right-hand side of (\ref{zOfT}).

Rather, the behavior we expect is the following. If we write the roots
$\{z_{j}(t)\}_{j=1}^{N}$ of $q_{t}$ in the form $z_{j}(t)=x_{j}(t)+iy_{j}(t),$
then we expect to have the approximate equalities
\begin{align*}
x_{j}(t)  &  \approx x_{j}(0)-\frac{t}{2}x_{j}(0)\\
y_{j}(t)  &  \approx c_{j}t\sqrt{1-\frac{1}{4}x_{j}(0)^{2}},
\end{align*}
where $c_{j}$ is a \textit{random} constant uniformly distributed between $-1$
and $1.$ This predicted behavior can be understood as the limiting case of
Conjecture \ref{add1.conj} with $X_{0}^{N}=0,$ $s=1,$ and $\tau_{0}$ tending
to zero. See Figure \ref{traj2.fig}, along with Slide 7 in the supplemental document.%

%TCIMACRO{\FRAME{ftbpFU}{2.9776in}{2.5434in}{0pt}{\Qcb{Plot of 30 of the curves
%$z_{j}(t)$ in Example \ref{ssToCirc.example}, with $N=1,000$ and $z_{j}(0)$
%close to $1.2.$ When $t=1,$ the distribution of points resembles a uniform
%distribution along the vertical segment in the unit circle with $x$-coordinate
%$0.6.$}}{\Qlb{traj2.fig}}{traj2.pdf}{\special{ language "Scientific Word";
%type "GRAPHIC";  maintain-aspect-ratio TRUE;  display "USEDEF";
%valid_file "F";  width 2.9776in;  height 2.5434in;  depth 0pt;
%original-width 5.0004in;  original-height 4.9865in;  cropleft "0";
%filename 'traj2.pdf';file-properties "XNPEU";}} }%
%BeginExpansion
\begin{figure}[ptb]%
\centering
\includegraphics[
%%natheight=4.986500in,
%%natwidth=5.000400in,
height=2.5434in,
width=2.9776in
]%
{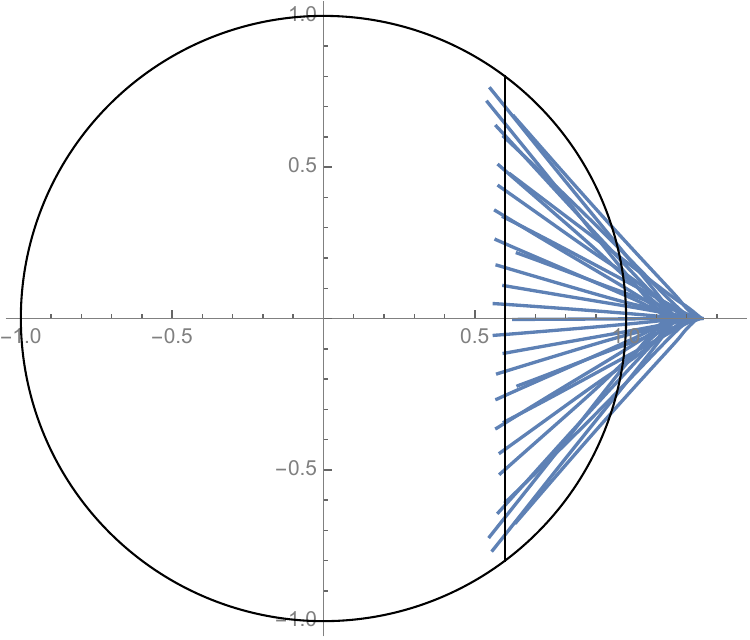}%
\caption{Plot of 30 of the curves $z_{j}(t)$ in Example \ref{ssToCirc.example}%
, with $N=1,000$ and $z_{j}(0)$ close to $1.2.$ When $t=1,$ the distribution
of points resembles a uniform distribution along the vertical segment in the
unit circle with $x$-coordinate $0.6.$}%
\label{traj2.fig}%
\end{figure}
%EndExpansion

As our third example of Conjecture \ref{add1.conj}, we consider the case in
which $X_{0}^{N}$ is Hermitian and $s=\tau_{0}=1.$ In that case, the limiting
eigenvalue distribution of $X_{0}^{N}+Z_{1,1}^{N}$ is computed in
\cite[Section 3]{HZ} and the limiting eigenvalue distribution of $X_{0}%
^{N}+Z_{1,\tau}^{N}$ is computed in \cite[Section 6]{Zhong2} (or the additive
version of \cite{HHmult}). We then further specialize to the case in which the
limiting eigenvalue distribution of $X_{0}^{N}$ has half of its mass at 1 and
half of its mass at $-1.$ Figure \ref{2taus.fig} shows the similarity between
the eigenvalues of $X_{0}^{N}+Z_{1,\tau}^{N}$ and the roots of the heat
evolution of the characteristic polynomial of $X_{0}^{N}+Z_{1,\tau_{0}}^{N}.$
See also Slide 8 in the supplemental document for an animation.

\subsection{Rigorous support for Conjecture \ref{add1.conj}: a deformation
result for second moments\label{secMom.sec}}

In this subsection, we present a rigorous result (Theorem
\ref{deformationAdditive.thm}) about the second moments of characteristic
polynomials of random matrices. We also give a heuristic argument for why
Theorem \ref{deformationAdditive.thm} should imply Conjecture \ref{add1.conj}.
This argument, however, requires concentration results that we currently do
not know how to establish.

Suppose $p^{N}$ is a family of random polynomials of degree $N$. In our
applications, we will take $p^{N}$ to be \textit{either} the characteristic
polynomial of a random matrix \textit{or} the heat evolution of such a
characteristic polynomial. We then consider the \textbf{second moment}
$D^{N}:\mathbb{C}\rightarrow\lbrack0,\infty)$ of $p^{N},$ defined as
\begin{equation}
D^{N}(z)=\mathbb{E}\{\left\vert p^{N}(z)\right\vert ^{2}%
\}.\label{secondMoment}%
\end{equation}
We then define%
\[
T^{N}(z)=\frac{1}{N}\log D^{N}(z)=\frac{1}{N}\log[\mathbb{E}\{\left\vert
p^{N}(z)\right\vert ^{2}\}]
\]
We expect to be able to recover the limiting root distribution of $p^{N}$ from
the function $D^{N}$ as follows:%
\begin{equation}
\text{limiting root distribution of }p^{N}=\frac{1}{4\pi}\Delta_{z}\left(
\lim_{N\rightarrow\infty}T^{N}(z)\right)  .\label{TNexpected}%
\end{equation}
That is to say, we expect that the large-$N$ limit of $T^{N}$ will be the log
potential of the limiting root distribution of $p^{N}.$ This sort of claim is
taken for granted in the physics literature (e.g., \cite{BGNTW1,BGNTW2}) but
it is not a rigorous result.%

%TCIMACRO{\FRAME{ftbpFU}{4.1554in}{2.009in}{0pt}{\Qcb{The $(s,\tau_{0})$
%eigenvalues (top), the $(s,\tau)$ eigenvalues (bottom left), and the
%$(\tau-\tau_{0})$-evolution of the $(s,\tau_{0})$ eigenvalues (bottom right),
%for the case that $X_{0}^{N}$ is Hermitian with eigenvalues equally
%distributed between $-1$ and $1.$ Shown for $s=1,$ $\tau_{0}=1,$ and
%$\tau=i/2$}}{\Qlb{2taus.fig}}{2taus.pdf}%
%width 4.1554in;  height 2.009in;  depth 0pt;  original-width 8.3333in;
%original-height 2.2641in;  cropleft "0";  croptop "1";  cropright "1";
%cropbottom "0";  filename '2taus.pdf';file-properties "XNPEU";}} }%
%BeginExpansion
\begin{figure}[ptb]%
\centering
\includegraphics[
%%natheight=2.264100in,
%%natwidth=8.333300in,
height=2.009in,
width=4.1554in
]%
{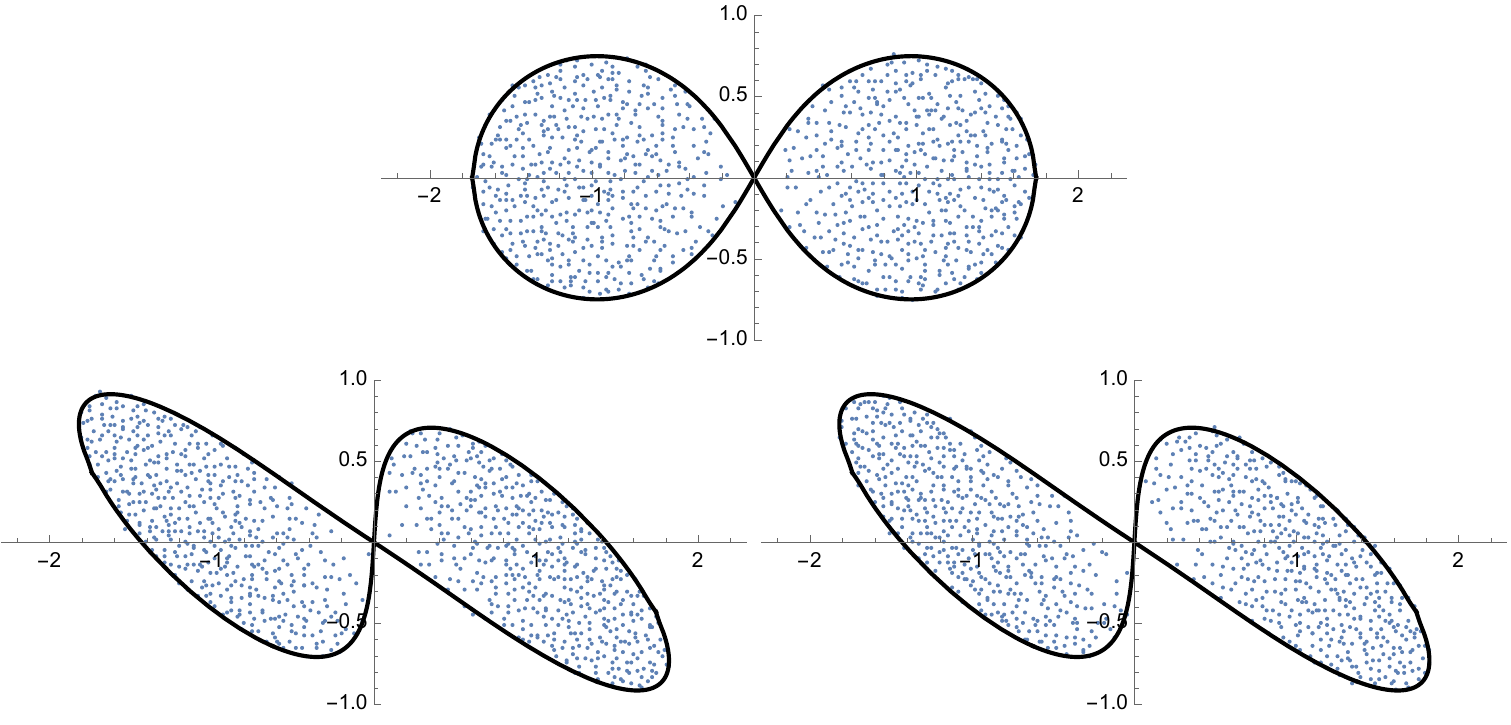}%
\caption{The $(s,\tau_{0})$ eigenvalues (top), the $(s,\tau)$ eigenvalues
(bottom left), and the $(\tau-\tau_{0})$-evolution of the $(s,\tau_{0})$
eigenvalues (bottom right), for the case that $X_{0}^{N}$ is Hermitian with
eigenvalues equally distributed between $-1$ and $1.$ Shown for $s=1,$
$\tau_{0}=1,$ and $\tau=i/2$}%
\label{2taus.fig}%
\end{figure}
%EndExpansion

To understand the claim in (\ref{TNexpected}), consider another function
obtained by interchanging the expectation value with the logarithm in the
formula for $T^{N},$ namely%
\begin{align*}
S^{N}(z)  &  =\frac{1}{N}\mathbb{E}\{\log\left\vert p^{N}(z)\right\vert
^{2}\}\\
&  =\mathbb{E}\left\{  \frac{1}{N}\sum_{j=1}^{N}\log\left\vert z-z_{j}%
^{N}\right\vert ^{2}\right\}  ,
\end{align*}
where $\{z_{j}^{N}\}_{j=1}^{N}$ is the collection of roots of $p^{N},$ listed
with their multiplicity. Then $\frac{1}{4\pi}\Delta_{z}S^{N}$ is easily seen
to be the \textit{expected empirical root distribution} of $Z^{N}.$ (Put the
Laplacian inside the expectation value and use that $\frac{1}{4\pi}%
\log\left\vert z\right\vert ^{2}$ is the Green's function for the Laplacian on
the plane.)

Suppose now that we have, as usual, a \textbf{concentration phenomenon}, in
which the root distribution of $p^{N}$ is approaching a deterministic limit as
$N$ goes to infinity. (In the case of characteristic polynomials of random
matrices, see, for example, Sections 2.3 and 4.4 in \cite{AGZ}.) In that case,
the large-$N$ limit of the \textit{expected} empirical root distribution
should be the almost sure limit of the eigenvalue distribution itself. In that
case,
\[
\text{limiting root distribution of }p^{N}=\frac{1}{4\pi}\Delta_{z}\left(
\lim_{N\rightarrow\infty}S^{N}(z)\right)  .
\]
But the same concentration phenomenon suggests that interchanging the
expectation value with the logarithm should not have much effect, so that
$S^{N}$ and $T^{N}$ should be almost equal. That is to say, if the empirical
measure of the set $\{z_{j}\}_{j=1}^{N}$ is, with high probability, close to a
deterministic measure $\mu,$ then \textit{both} $S^{N}$ and $T^{N}$ should be
close to the log potential of $\mu,$ and we should have
\begin{align*}
\text{limiting root distribution of }p^{N}  &  =\frac{1}{4\pi}\Delta
_{z}\left(  \lim_{N\rightarrow\infty}S^{N}(z)\right) \\
&  =\frac{1}{4\pi}\Delta_{z}\left(  \lim_{N\rightarrow\infty}T^{N}(z)\right)
,
\end{align*}
confirming (\ref{TNexpected}).

We consider the second moments (as in (\ref{secondMoment})) of the
characteristic polynomials of the random matrix models introduced in the
previous section, starting from the additive case. Consider an $N\times N$
\textquotedblleft elliptic\textquotedblright\ random matrix $Z_{s,\tau}^{N}$
with parameters $s$ and $\tau,$ as in (\ref{ZNdef}) and (\ref{sAndTau}). Take
another random matrix $X_{0}^{N}$ that is independent of $Z_{s,\tau}^{N}$ and
define a function $D^{N}$ by
\begin{equation}
D^{N}(z,s,\tau)=\mathbb{E}\{\left\vert \det(zI-(X_{0}^{N}+Z_{s,\tau}%
^{N}))\right\vert ^{2}\}. \label{DNdef}%
\end{equation}
Here $z$ ranges over $\mathbb{C}$, $s$ ranges over $(0,\infty),$ and $\tau$
ranges over the set of complex numbers with $\left\vert \tau-s\right\vert <s.$
Of course, this function depends also on the distribution of the random matrix
$X_{0}^{N}$ but we suppress this dependence in the notation.

We now come to the main theorem supporting the additive heat flow conjecture
(Conjecture \ref{add1.conj}).

\begin{theorem}
[Deformation theorem for second moment]\label{deformationAdditive.thm}Suppose
$\tau_{0}$ and $\tau$ are complex numbers satisfying $\left\vert \tau
_{0}-s\right\vert \leq s$ and $\left\vert \tau-s\right\vert \leq s,$ in
accordance with (\ref{tauIneq}). Consider random matrices $X_{0}^{N}%
+Z_{s,\tau_{0}}^{N}$ and $X_{0}^{N}+Z_{s,\tau}^{N},$ respectively, where
$X_{0}^{N}$ is independent of $Z_{s,\tau_{0}}^{N}$ and $Z_{s,\tau}^{N}$ but
not necessarily Hermitian. Then the function $D^{N}$ in (\ref{DNdef}) can also
be computed as%
\begin{equation}
D^{N}(z,s,\tau)=\mathbb{E}\left\{  \left\vert \exp\left(  \frac{(\tau-\tau
_{0})}{2N}\frac{\partial^{2}}{\partial z^{2}}\right)  \det(zI-(X_{0}%
^{N}+Z_{s,\tau_{0}}^{N}))\right\vert ^{2}\right\}  . \label{Dsecond}%
\end{equation}

\end{theorem}

The proposition says that one can compute $D^{N}$ in two different ways:
first, according to the definition (\ref{DNdef}), by taking the expectation of
the magnitude-squared of the characteristic polynomial of the $(s,\tau
)$-model; or, second, by applying the heat operator $\exp\left(  \frac
{(\tau-\tau_{0})}{N}\frac{\partial^{2}}{\partial z^{2}}\right)  $ to the
characteristic polynomial of the $(s,\tau_{0})$-model and then taking the
expectation of magnitude-square of this new polynomial.

In the notation of Conjecture \ref{add1.conj}, Theorem
\ref{deformationAdditive.thm} says that the random polynomials $q_{s,\tau
_{0},\tau}$ and $p_{s,\tau}$ \textit{have the same second moments}. Since we
expect to recover the limiting root distributions of $q_{s,\tau_{0},\tau}$ and
$p_{s,\tau}$ from their second moments, as in (\ref{TNexpected}), Theorem
\ref{deformationAdditive.thm} provides a strong motivation for Conjecture
\ref{add1.conj}.

Let us use the notation $\{z_{j}^{s,\tau}\}_{j=1}^{N}$ for the roots of
$p_{s,\tau}$ and the notation $\{z_{j}^{s,\tau_{0}}(\tau)\}_{j=1}^{N}$ for the
roots of $q_{s,\tau_{0},\tau}\,$(i.e., the roots of $p_{s,\tau_{0}},$ evolved
for time $\tau$ by the heat flow). Then Theorem \ref{deformationAdditive.thm}
can be restated as%
\[
\mathbb{E}\left\{  \left\vert \prod_{j=1}^{N}(z-z_{j}^{s,\tau})\right\vert
^{2}\right\}  =\mathbb{E}\left\{  \left\vert \prod_{j=1}^{N}(z-z_{j}%
^{s,\tau_{0}}(\tau))\right\vert ^{2}\right\}  .
\]

\subsection{Proof of Theorem \ref{deformationAdditive.thm}}

The theorem will following easily from the following result, which is of
independent interest.

\begin{proposition}
\label{DPDE.prop}The function $D^{N}$ in (\ref{DNdef}) satisfies the
second-order linear PDEs%
\begin{align}
\frac{\partial D^{N}}{\partial\tau}  &  =\frac{1}{2N}\frac{\partial^{2}D^{N}%
}{\partial z^{2}}\label{dDdTau}\\
\frac{\partial D^{N}}{\partial\bar{\tau}}  &  =\frac{1}{2N}\frac{\partial
^{2}D^{N}}{\partial\bar{z}^{2}} \label{dDdTauBar}%
\end{align}
for $\tau$ in the set $\left\vert \tau-s\right\vert <s.$
\end{proposition}

In the proposition, it is understood that the distribution of the random
matrix $X_{0}^{N}$ is fixed (independent of $\tau$) for each $N.$

We establish two lemmas that will be used to prove the proposition. The first
gives the PDE\ satisfied by the probability density of the random matrix
$Z_{s,\tau}^{N}.$ The second relates derivatives of a characteristic
polynomial in the matrix variable to derivatives in the complex variable.

Consider $M_{N}(\mathbb{C})$ as a real vector space of dimension $2N^{2},$
equipped with the real-valued inner product $\left\langle \cdot,\cdot
\right\rangle _{N}$ given by the scaled Hilbert--Schmidt inner product:
\[
\left\langle Z,W\right\rangle _{N}=N\operatorname{Re}[\operatorname{Trace}%
(Z^{\ast}W)].
\]
(The density of the Ginibre ensemble is a constant times $e^{-\left\langle
Z,Z\right\rangle _{N}}.$) We choose an orthonormal basis $\{X_{j}%
\}_{j=1}^{N^{2}}\cup\{Y_{j}\}_{j=1}^{N^{2}}$ such that $X_{j}$ is Hermitian
and $Y_{j}=iX_{j}.$ We then form the translation-invariant differential
operators $\tilde{X}_{j}$, $\tilde{Y}_{j},$ $Z_{j},$ and $\bar{Z}_{j}$ on
$M_{N}(\mathbb{C})$ as%
\begin{align}
\tilde{X}_{j}f(A)  &  =\left.  \frac{d}{du}f(A+uX_{j})\right\vert _{u=0}%
;\quad\tilde{Y}_{j}f(A)=\left.  \frac{d}{du}f(A+uY_{j})\right\vert
_{u=0};\nonumber\\
Z_{j}  &  =\frac{1}{2}(\tilde{X}_{j}-i\tilde{Y}_{j});\quad\bar{Z}_{j}=\frac
{1}{2}(\tilde{X}_{j}+i\tilde{Y}_{j}). \label{ZjZjBar}%
\end{align}
We then introduce operators $\Delta,$ $\partial^{2},$ and $\bar{\partial}^{2}$
given by%
\begin{equation}
\Delta=\sum_{j=1}^{N^{2}}\tilde{X}_{j}^{2};\quad\partial^{2}=\sum_{j=1}%
^{N^{2}}Z_{j}^{2};\quad\bar{\partial}^{2}=\sum_{j=1}^{N^{2}}\bar{Z}_{j}.
\label{threeOps}%
\end{equation}

\begin{lemma}
\label{firstID.lem}Let $\gamma_{s,\tau}^{N}$ be the density of the law of the
random matrix $Z_{s,\tau}^{N}$ defined in Section \ref{additiveConjecture.sec}%
. Then $\gamma_{s,\tau}^{N}$ satisfies the PDE%
\begin{equation}
\frac{\partial\gamma_{s,\tau}^{N}}{\partial\tau}=-\frac{1}{2}\partial
^{2}\gamma_{s,\tau}^{N}. \label{gammaPDE}%
\end{equation}
Furthermore, if $f$ is a (not necessarily holomorphic) polynomial function on
the space $M_{N}(\mathbb{C})$ of all $N\times N$ matrices with complex
entries, we have%
\begin{equation}
\frac{\partial}{\partial\tau}\mathbb{E}\{f(X_{0}^{N}+Z_{s,\tau}^{N}%
)\}=-\frac{1}{2}\mathbb{E}\{\partial^{2}f(X_{0}^{N}+Z_{s,\tau}^{N})\}.
\label{diffExpect}%
\end{equation}

\end{lemma}

\begin{proof}
Let $\Gamma_{s,\tau}^{N}$ be the Gaussian measure on $M_{N}(\mathbb{C})$
describing the law of $Z_{s,\tau}^{N}.$ It is given by
\[
\gamma_{s,\tau}^{N}=\exp\left\{  \frac{1}{2}\Delta_{s,\tau}\right\}
(\delta_{0})
\]
where $\delta_{0}$ is a $\delta$-function at the origin and where
$\Delta_{s,\tau}$ is defined as%
\begin{equation}
\Delta_{s,\tau}=s\Delta-\tau\partial^{2}-\bar{\tau}\bar{\partial}^{2}.
\label{DeltaStau}%
\end{equation}
(The formula (\ref{DeltaStau}) is equivalent to Eq. (1.7) in \cite{DHKcomplex}%
; see also the equation between Eqs. (1.13) and (1.14) in \cite{DHKcomplex}.)
This operator is elliptic precisely when $\left\vert \tau-s\right\vert <s$ and
semi-elliptic in the borderline case $\left\vert \tau-s\right\vert =s.$ The
PDE (\ref{gammaPDE}) follows formally from the commutativity of the operators
on the right-hand side of (\ref{DeltaStau}). Rigorously, it is
straightforward, if not terribly illuminating, to verify (\ref{gammaPDE}) from
the explicit formula for $\gamma_{s,\tau}^{N}$ in \cite[Eq. (1.15)]%
{DHKcomplex}.

Meanwhile, since $X_{0}^{N}$ is independent of $Z_{s,\tau}^{N},$ we have%
\begin{equation}
\mathbb{E}\{f(X_{0}^{N}+Z_{s,\tau}^{N})\}=\int\int f(a+b)\gamma_{s,\tau}%
^{N}(b)~db~d\nu(a), \label{intAB}%
\end{equation}
where $db$ is the Lebesgue measure on $M_{N}(\mathbb{C})$ and $\nu$ is the law
of $X_{0}^{N}.$ Now, in Section \ref{additiveConjecture.sec}, we assume that
all of the expected $\ast$-moments of $X_{0}^{N}$ are finite, which in
particular guarantees that the integral on the right-hand side of
(\ref{intAB}) is convergent. It is then not hard to justify differentiating
with respect to $\tau$ under the integral on the right-hand side of
(\ref{intAB}). Using (\ref{gammaPDE}), we obtain%
\begin{equation}
\frac{\partial}{\partial\tau}\mathbb{E}\{f(X_{0}^{N}+Z_{s,\tau}^{N}%
)\}=-\frac{1}{2}\int\int f(a+b)\partial^{2}\gamma_{s,\tau}^{N}(b)~db~d\nu(a).
\label{intAB2}%
\end{equation}
Then since $\gamma_{s,\tau}^{N}$ and its derivatives have rapid decay at
infinity and $f$ is a polynomial, we can integrate by parts in the inner
integral on the right-hand side of (\ref{intAB2}) to obtain%
\begin{equation}
\frac{\partial}{\partial\tau}\mathbb{E}\{f(X_{0}^{N}+Z_{s,\tau}^{N}%
)\}=-\frac{1}{2}\int\int\partial^{2}f(a+b)\gamma_{s,\tau}^{N}(b)~db~d\nu(a).
\label{intAB3}%
\end{equation}
Here, the operator $\partial^{2}$ on the right-hand side should in principle
be acting in the $b$ variable. But since $\partial^{2}$ is a
translation-invariant operator, this is the same as first applying
$\partial^{2}$ to the function $f$ and then evaluating at $a+b.$

Alternatively, we may appeal to (the commutative case of) Proposition 4.7 in
\cite{DHKcomplex}, which tells us that the inner integral in (\ref{intAB}) can
be computed (for polynomials functions $f$) as
\begin{equation}
\int f(a+b)\gamma_{s,\tau}^{N}(b)=\left.  \left(  \exp\left\{  \frac{1}%
{2}\Delta_{s,\tau}\right\}  f(a+b)\right)  \right\vert _{b=0}, \label{intAB4}%
\end{equation}
where the exponential is computed as a convergent power series in powers of
$\Delta_{s,\tau}.$ At this point, the (commuting!) operators that make up
$\Delta_{s,\tau}$ in (\ref{DeltaStau}) become operators in some
\textit{finite-dimensional} space of polynomials of degree at most $k.$ It is
therefore permissible to differentiate $\exp\left\{  \frac{1}{2}\Delta
_{s,\tau}\right\}  $ with respect to $\tau$ in the obvious way, pulling down a
factor of $-\frac{1}{2}\partial^{2}$ onto $f.$ We may then apply
(\ref{intAB4}) in the opposite direction and we will arrive at (\ref{intAB3}).
\end{proof}

\begin{lemma}
\label{secondID.lem}Let $A$ be a variable ranging over $M_{N}(\mathbb{C})$,
let $z$ be a variable ranging over $\mathbb{C}$, and let $B$ be a fixed
matrix. Then%
\begin{equation}
\partial^{2}\det(zI-(B+A))=-\frac{1}{N}\frac{d^{2}}{dz^{2}}\det(zI-(B+A)),
\label{Del2dz2}%
\end{equation}
where $\partial^{2}$ is the operator defined in (\ref{threeOps}).
\end{lemma}

\begin{proof}
Both sides of (\ref{Del2dz2}) are easily seen to be polynomials in $z$ and the
entries of $A.$ It therefore suffices to establish the identity on the
nonempty open set of pairs $(A,z)$ where $zI-A$ is invertible. We will use the
notation%
\[
Q=(zI-(B+A))
\]%
\[
R=Q^{-1}%
\]
and we can verify the following basic rules for computing:%
\[
Z_{j}A=X_{j},
\]
where $Z_{j}$ is as in (\ref{ZjZjBar}). We then use Jacobi's formula
\cite[Section 8.3]{MagNeu} for the derivative of the determinant (invertible
case):%
\[
\frac{d}{du}\det(A(u))=\det(A(u))\mathrm{Trace}\left(  A^{-1}\frac{dA}%
{du}\right)  ,
\]
where $\mathrm{Trace}$ is the ordinary trace; and the formula \cite[Section
8.4]{MagNeu} for the derivative of the inverse:%
\[
\frac{d}{du}A(u)^{-1}=-A(u)^{-1}\frac{dA}{du}A(u)^{-1}.
\]

In what follows, we use the normalized trace, $\mathrm{tr}[\cdot
]=\mathrm{Trace}[\cdot]/N.$ Using the tools in the previous paragraph, we can
easily compute%
\begin{align*}
Z_{j}\det(Q)  &  =-N\det(Q)\mathrm{tr}[RX_{j}]\\
Z_{j}^{2}\det(Q)  &  =N^{2}\det(Q)\mathrm{tr}[RX_{j}]^{2}-N\det(Q)\mathrm{tr}%
[RX_{j}RX_{j}].
\end{align*}
We then use the \textquotedblleft magic formulas\textquotedblright\ for sums
involving an orthonormal basis $\{X_{j}\}_{j=1}^{N^{2}}$ as in the proof of
Lemma \ref{firstID.lem}:%
\begin{align}
\sum_{j}X_{j}AX_{j}  &  =\operatorname{tr}[A]I\label{magic1}\\
\sum_{j}\operatorname{tr}[X_{j}A]\operatorname{tr}[X_{j}B]  &  =\frac{1}%
{N^{2}}\operatorname{tr}[AB]. \label{magic2}%
\end{align}
(See, for example, \cite[Proposition 3.1]{DHKlargeN}, with a change of sign
because the $X_{j}$'s here are Hermitian rather than skew-Hermitian.) We
therefore obtain:%
\begin{equation}
\partial^{2}\det(Q)=\det(Q)\mathrm{tr}[R^{2}]-N\det(Q)\mathrm{tr}[R]^{2}.
\label{DelSquaredDet}%
\end{equation}

Meanwhile, a similar computation shows
\begin{align}
\frac{d}{dz}\det(Q)  &  =N\det(Q)\mathrm{tr}[R]\label{dzDet}\\
\frac{d^{2}}{dz^{2}}\det(Q)  &  =N^{2}\det(Q)\mathrm{tr}[R]^{2}-N\det
(Q)\mathrm{tr}[R^{2}]. \label{dzSquaredDet}%
\end{align}
Comparing (\ref{DelSquaredDet}) to (\ref{dzSquaredDet}) gives the claimed result.
\end{proof}

We these lemmas in hand, we are ready for the proof of Proposition
\ref{DPDE.prop}.

\begin{proof}
[Proof of Proposition \ref{DPDE.prop}]We apply the identity (\ref{diffExpect})
in Lemma \ref{firstID.lem} to the function $f$ given by%
\[
f(A)=\left\vert \det(zI-(B+A))\right\vert ^{2}.
\]
We note that $\partial^{2}$ treats antiholomorphic functions as constants. We
thus obtain%
\begin{align}
\frac{\partial D^{N}}{\partial\tau}  &  =-\frac{1}{2}\mathbb{E}\{\partial
^{2}\left\vert \det(zI-(X_{0}^{N}+Z_{s,\tau}^{N}))\right\vert ^{2}%
\}\nonumber\\
&  =-\frac{1}{2}\mathbb{E}\left\{  \overline{\det(zI-(X_{0}^{N}+Z_{s,\tau}%
^{N}))}~\partial^{2}\det(zI-(X_{0}^{N}+Z_{s,\tau}^{N}))\right\}  .
\label{DNdiffAdd}%
\end{align}
Using Lemma \ref{secondID.lem}, (\ref{DNdiffAdd}) becomes%
\begin{align*}
\frac{\partial D^{N}}{\partial\tau}  &  =\frac{1}{2N}\mathbb{E}\left\{
\overline{\det(zI-A)}~\frac{\partial^{2}}{\partial z^{2}}\det(zI-A)\right\} \\
&  =\frac{1}{2N}\frac{\partial^{2}}{\partial z^{2}}D_{\tau}^{N},
\end{align*}
as claimed in (\ref{dDdTau}). Since $D^{N}$ is real-valued, (\ref{dDdTauBar}) follows.
\end{proof}

Now for the proof of Theorem \ref{deformationAdditive.thm}.

\begin{proof}
[Proof of Theorem \ref{deformationAdditive.thm}]Let $\tilde{D}_{\tau_{0}}%
^{N}(s,\tau,z)$ denote the function on the right-hand side of (\ref{Dsecond}),
so that when $\tau=\tau_{0},$ we have $\tilde{D}_{\tau_{0}}^{N}(s,\tau
_{0},z)=D^{N}(s,\tau_{0},z).$ Our goal is to show that $\tilde{D}_{\tau_{0}%
}^{N}=D^{N}.$ The function $\tilde{D}_{\tau_{0}}^{N}$ can be computed as%
\begin{align}
&  \tilde{D}_{\tau_{0}}^{N}(s,\tau,z)\nonumber\\
&  =\mathbb{E}\left\{  \exp\left(  \frac{1}{2N}\left(  (\tau-\tau_{0}%
)\frac{\partial^{2}}{\partial z^{2}}+(\bar{\tau}-\bar{\tau}_{0})\frac
{\partial^{2}}{\partial\bar{z}^{2}}\right)  \right)  \left\vert \det
(zI-(X_{0}^{N}+Z_{s,\tau_{0}}^{N}))\right\vert ^{2}\right\} \nonumber\\
&  =\exp\left(  \frac{1}{2N}\left(  (\tau-\tau_{0})\frac{\partial^{2}%
}{\partial z^{2}}+(\bar{\tau}-\bar{\tau}_{0})\frac{\partial^{2}}{\partial
\bar{z}^{2}}\right)  \right)  \mathbb{E}\left\{  \left\vert \det(zI-(X_{0}%
^{N}+Z_{s,\tau_{0}}^{N}))\right\vert ^{2}\right\}  . \label{DNN}%
\end{align}

From the last expression in (\ref{DNN}), we can see that $\tilde{D}_{\tau_{0}%
}^{N}$ satisfies the same PDEs (\ref{dDdTau}) and (\ref{dDdTauBar}) as
$D^{N}(s,\tau,z),$ as a function of $\tau$ and $z$. Thus,
\begin{equation}
\tilde{D}_{\tau_{0}}^{N}(s,\tau_{0}+t(\tau-\tau_{0}),\tau_{0},z)\text{ and
}D^{N}(s,\tau_{0}+t(\tau-\tau_{0}),z) \label{twoFunctions}%
\end{equation}
will satisfy the same PDE in $t$ and $z$ for $0\leq t\leq1,$ with equality at
$t=0.$ Since both functions are, for all values of the other variables,
polynomials in $z$ and $\bar{z}$ of degree $2N,$ the PDE in $t$ and $z$ is
actually an ODE with values in a finite-dimensional vector space. Thus, by
uniqueness of solutions of ODEs, we conclude that the two functions in
(\ref{twoFunctions}) are equal for all $t$; setting $t=1$ gives the claimed result.
\end{proof}

\subsection{A \textquotedblleft counterexample\textquotedblright}

Example \ref{ssToCirc.example} describes what we expect to happen if we evolve
the characteristic polynomial of a GUE matrix by the forward heat flow
$\exp\{\frac{t}{2N}\frac{d^{2}}{dz^{2}}\}$. The roots start out at $t=0$ with
an asymptotically semicircular distribution on $[-2,2].$ Then for $0<t<2,$ the
roots should be asymptotically uniform on an ellipse with semi-axes $2-t$ and
$t.$

One might then expect that if $p_{0}^{N}$ is \textit{any} sequence of
real-rooted polynomials of degree $N$ having an asymptotically semicircular
distribution of roots on $[-2,2],$ the heat-evolved polynomials would have the
same limiting root distribution as in the previous paragraph. But this claim
is false. A counterexample is provided by the (scaled) Hermite polynomials%
\[
H_{N}(z)=\exp\left\{  -\frac{1}{2N}\frac{d^{2}}{dz^{2}}\right\}  z^{N}.
\]
These polynomials are known \cite{CalPer} to have an asymptotically
semicircular distribution of roots on $[-2,2].$ But if we apply the forward
heat equation, we get%
\[
\exp\left\{  \frac{t}{2N}\frac{d^{2}}{dz^{2}}\right\}  H_{N}(z)=\exp\left\{
-\frac{(1-t)}{2N}\frac{d^{2}}{dz^{2}}\right\}  z^{N}.
\]
For $0<t<1,$ the above polynomials are just rescaled Hermite polynomials and
have an asymptotically semicircular distribution on $[-2\sqrt{1-t},2\sqrt
{1-t}].$

Of course, the result in the preceding paragraph does not actually contradict
the conjecture in Example \ref{ssToCirc.example}, which is only about the
characteristic polynomial of GUE matrices. Nor does it contradict Conjecture
\ref{introConj.conj}, since the Cauchy transform of the semicircular law is
discontinuous on the interval $[-2,2].$ It is nevertheless notable that such
different behavior can arise from two sequences of polynomials with the same
limiting root distribution. We attribute this difference in behavior to the
very evenly spaced nature of the zeros of the Hermite polynomials, which
contrasts with the random fluctuations in the spacings that occurs for the
eigenvalues of a GUE matrix.

Although the (rigorous) behavior of a heat-evolved Hermite polynomial
contrasts with the (conjectural) behavior of the characteristic polynomial of
a GUE, there is one common aspect to the two. Consider the holomorphic moments
$M_{a,N}$ of a sequence $p^{N}$ of polynomials of degree $N,$ namely%
\[
M_{a,N}=\frac{1}{N}\sum_{j=1}^{N}(z_{j}^{N})^{a},\quad a=0,1,2,\ldots,
\]
where $\{z_{j}^{N}\}_{j=1}^{N}$ are the roots of $p^{N}.$ According to Theorem
\ref{holoMoment.thm} in Section \ref{moments.sec}, the large-$N$ behavior of
the holomorphic moments of heat-evolved polynomials is completely determined
by the large-$N$ behavior of the holomorphic moments of the initial
polynomials. Thus, heat-evolved Hermite polynomials and heat-evolved
characteristic polynomials of a GUE will have the same limiting holomorphic
moments. These moments, however, do not uniquely determine the limiting root
distribution in the case that the roots are complex.

\section{A dynamical systems perspective\label{dynamical.sec}}

In this section, we record the system of differential equations satisfied by
the roots of a polynomial as the polynomial itself evolves according to the
heat equation. We then use this result to argue for the push-foward result in
Point \ref{push.point} of Conjecture \ref{introConj.conj}. A similar argument
holds in the multiplicative case.

\subsection{The dynamics of the roots\label{dynAdd.sec}}

We now record a well-known (and elementary) result about the evolution of
zeros of a polynomial undergoing the heat flow.

\begin{proposition}
\label{ODEadd.prop}Let $p_{0}^{N}$ be a polynomial of degree $N$ and let
$p_{\tau}^{N}$ be the heat evolution of $p_{0}^{N},$ as defined in
(\ref{pTauDef}), so that
\begin{equation}
\frac{\partial p_{\tau}^{N}}{\partial\tau}=\frac{1}{2N}\frac{\partial
^{2}p_{\tau}^{N}}{\partial z^{2}}. \label{qTauPDE}%
\end{equation}

Suppose that, for some $\sigma\in\mathbb{C},$ the zeros of $p_{\sigma}$ are
distinct. Then for all $\tau$ in a neighborhood of $\sigma,$ it is possible to
order the zeros of $p_{\tau}^{N}$ as $z_{1}(\tau),\ldots,z_{N}(\tau)$ so that
each $z_{j}(\tau)$ depends holomorphically on $\tau$ and so that the
collection $\{z_{j}(\tau)\}_{j=1}^{N}$ satisfies the following system of
holomorphic differential equations:%
\begin{equation}
\frac{dz_{j}(\tau)}{d\tau}=-\frac{1}{N}\sum_{k\neq j}\frac{1}{z_{j}%
(\tau)-z_{k}(\tau)}. \label{theODE}%
\end{equation}
The paths $z_{j}(\tau)$ then satisfy%
\begin{equation}
\frac{d^{2}z_{j}(\tau)}{d\tau^{2}}=-\frac{2}{N^{2}}\sum_{k\neq j}\frac
{1}{(z_{j}(\tau)-z_{k}(\tau))^{3}}. \label{secondDeriv}%
\end{equation}

\end{proposition}

The sums on the right-hand side of (\ref{theODE}) and (\ref{secondDeriv}) are
over all $k$ different from $j,$ with $j$ fixed. The result in (\ref{theODE})
is discussed on Terry Tao's blog \cite{Tao1} and dates back at least to the
work of Csordas, Smith, and Varga \cite[Lemma 2.4]{CSV}. The result that
(\ref{theODE}) implies (\ref{secondDeriv}) dates back at least to work of
Choodnovsky and Choodnovsky \cite[Corollary 7]{CC}. We emphasize that the
polynomial $p_{\tau}^{N}$ in (\ref{pTauDef}) is well-defined for all $\tau
\in\mathbb{C}$ by (\ref{pTauDef}), whether the roots are distinct or not.

\begin{remark}
\label{cm.remark}The second-order equations in (\ref{secondDeriv}) are the
equations of motion for the \textbf{rational Calogero--Moser system}. (Take
$\omega=0$ and $g^{2}=-1/N$ in the notation of \cite[Eq. (3)]{Cal}.) It
follows that solutions to (\ref{theODE}) are special cases of solutions to the
rational Calogero--Moser system, in which the initial velocities are chosen to
satisfy (\ref{theODE}) at $\tau=0.$
\end{remark}

Remark \ref{cm.remark} (together with Conjecture \ref{add1.conj}) indicates a
novel connection between integrable systems and random matrix theory. The
negative value of $g^{2}$ in the remark means that $\tau$ is real and all the
$z_{j}$'s are initially real, the system is attractive and collisions can
occur---allowing the points to move off the real line.

For completeness, we supply a proof of Proposition \ref{ODEadd.prop}.

\begin{proof}
[Proof of Proposition \ref{ODEadd.prop}]The local holomorphic dependence of
the roots on $\tau$ is an elementary consequence of the holomorphic version of
the implicit function theorem, with the assumption that the roots of
$p_{\sigma}$ are distinct guaranteeing that $dp_{\sigma}/dz$ is nonzero at
each root.

If $z_{j}$ is a simple zero of a polynomial $p,$ an easy power-series argument
shows that
\[
\frac{p^{\prime\prime}(z_{j})}{p^{\prime}(z_{j})}=2\left.  \left(
\frac{p^{\prime}(z)}{p(z)}-\frac{1}{z-z_{j}}\right)  \right\vert _{z=z_{j}},
\]
from which we easily obtain%
\begin{equation}
\frac{p^{\prime\prime}(z_{j})}{p^{\prime}(z_{j})}=2\sum_{k\neq j}\frac
{1}{z_{j}-z_{k}}, \label{pPrimePrime}%
\end{equation}
where the sum is over all $k$ different from $j,$ with $j$ fixed. We may then
differentiate the identity $p_{\tau}^{N}(z_{j}(\tau))=0$ to obtain%
\begin{equation}
\frac{\partial p_{\tau}^{N}}{\partial\tau}(z_{j}(\tau))+(p_{\tau}^{N}%
)^{\prime}(z_{j}(\tau))\frac{dz_{j}}{d\tau}=0. \label{diffId}%
\end{equation}
Using (\ref{qTauPDE}), (\ref{diffId}) gives%
\[
\frac{dz_{j}}{d\tau}=-\frac{\partial p_{\tau}^{N}}{\partial\tau}(z_{j}%
(\tau))\frac{1}{(p_{\tau}^{N})^{\prime}(z_{j}(\tau))}=-\frac{1}{2N}%
\frac{(p_{\tau}^{N})^{\prime\prime}(z_{j}(\tau))}{(p_{\tau}^{N})^{\prime
}(z_{j}(\tau))}.
\]
Applying (\ref{pPrimePrime}) then gives (\ref{theODE}).

For the second derivative, we suppress the dependence of $z_{j}$ on $\tau$ and
we use (\ref{theODE}) to make a preliminary calculation for each pair $j$ and
$k$ with $j\neq k$:
\begin{equation}
\frac{d}{d\tau}(z_{j}-z_{k})=-\frac{1}{N}\sum_{l\neq j}\frac{1}{z_{j}-z_{l}%
}+\frac{1}{N}\sum_{l\neq k}\frac{1}{z_{k}-z_{l}}. \label{dDiff1}%
\end{equation}
We then split each sum over $l$ on the right-hand side of (\ref{dDiff1}) into
a sum over $l\notin\{j,k\}$ plus an additional term:%
\begin{align*}
\frac{d}{d\tau}(z_{j}-z_{k})  &  =-\frac{1}{N}\left(  \frac{1}{z_{j}-z_{k}%
}+\sum_{l\notin\{j,k\}}\frac{1}{z_{j}-z_{l}}\right) \\
&  +\frac{1}{N}\left(  \frac{1}{z_{k}-z_{j}}+\sum_{l\notin\{j,k\}}\frac
{1}{z_{k}-z_{l}}\right)  ,
\end{align*}
which simplifies to
\begin{equation}
\frac{d}{d\tau}(z_{j}-z_{k})=-\frac{2}{N}\frac{1}{z_{j}-z_{k}}+\frac{1}{N}%
\sum_{l\notin\{j,k\}}\frac{z_{j}-z_{k}}{(z_{j}-z_{l})(z_{k}-z_{l})}.
\label{dDiff2}%
\end{equation}

We then differentiate (\ref{theODE}) using (\ref{dDiff2}) to get%
\begin{align*}
&  \frac{d^{2}z_{j}}{d\tau^{2}}\\
&  =\frac{1}{N}\sum_{k\neq j}\frac{1}{(z_{j}-z_{k})^{2}}\left(  -\frac{2}%
{N}\frac{1}{z_{j}-z_{k}}+\frac{1}{N}\sum_{l\notin\{j,k\}}\frac{z_{j}-z_{k}%
}{(z_{j}-z_{l})(z_{k}-z_{l})}\right) \\
&  =-\frac{2}{N^{2}}\sum_{k\neq j}\frac{1}{(z_{j}-z_{k})^{3}}+\frac{1}{N^{2}%
}\sum_{\substack{k,l:\\(j,k,l)~\text{distinct}}}\frac{1}{(z_{j}-z_{k}%
)(z_{j}-z_{l})(z_{k}-z_{l})}.
\end{align*}
The last sum over $k$ and $l$ is zero because the range of the sum is
invariant under interchange of $k$ and $l,$ but the summand changes sign under
interchange of $k$ and $l,$ leaving us with the claimed result.
\end{proof}

We now explain how Proposition \ref{ODEadd.prop} should imply push-forward
result in Conjecture \ref{introConj.conj}.

\begin{idea}
\label{AdditiveMotion.idea}Suppose $p_{0}^{N}$ is a sequence of polynomials of
degree $N$ whose limiting root distribution is a sufficiently regular
probability measure $\mu_{0}$ with compact support. Then we expect that the
following results will hold.

\begin{enumerate}
\item \label{initVelocity.point}The initial velocity of the roots will be
approximately the Cauchy transform of $\mu_{0}$ evaluated at the initial root:%
\[
z_{j}^{\prime}(0)\approx-m_{0}(z_{j}(0)).
\]

\item \label{secondDeriv.point}The second derivative of the roots will be
small at any time $\tau$ for which the distribution of roots approximates a
smooth two-dimensional distribution in the plane:%
\[
z_{j}^{\prime\prime}(\tau)\approx0.
\]

\end{enumerate}

If these results hold, then as in (\ref{approxMotion}),
\[
z_{j}(\tau)\approx z_{j}(0)-\tau m_{0}(z_{j}(0)),
\]
in which case the push-forward result in Point \ref{push.point} of Conjecture
\ref{introConj.conj} will hold.
\end{idea}

We now give a nonrigorous argument for the preceding idea.

\begin{proof}
[Argument for Idea \ref{AdditiveMotion.idea}]Point \ref{initVelocity.point} is
obtained by approximating the sum over $k$ in (\ref{theODE}), with the factor
of $1/N,$ by an integral with respect to $\mu_{0}.$ Meanwhile, the right-hand
side of (\ref{secondDeriv}) is formally of order $1/N,$ since it is a sum over
$N-1$ values but we are dividing by $N^{2}.$ In reality, the sum is not as
small as this naive calculation would suggest, because $1/z^{3}$ is not a
locally integrable function on the plane. Thus, there will be a substantial
contribution in (\ref{secondDeriv}) from the values of $z_{k}$ closest to
$z_{j}.$ But as long as these nearest neighbors are of order $1/\sqrt{N}%
$---which we expect if the distribution of roots is two-dimensional---we still
expect that the right-hand side of (\ref{secondDeriv}) will be of order
$1/\sqrt{N}.$ It is difficult to prove anything like this rigorously, in part
because it is difficult to control the nearest-neighbor spacings at time
$\tau,$ even if these spacings behave nicely at time zero.
\end{proof}

\subsection{Solving the equations using the Calogero--Moser systems}

In this subsection, we explain how the theory of Calogero--Moser systems can
be used to give an \textquotedblleft explicit\textquotedblright\ formula for
the roots of a heat-evolved polynomial. Of course, the formula still requires
a calculation: finding the eigenvalues of a matrix. Nevertheless, it is
notable that we can find the roots of the heat-evolved polynomial, for any one
fixed time $\tau,$ directly in a comparatively simple way.

Let $p_{0}^{N}$ be a polynomial of degree $N$ with roots $z_{1},\ldots,z_{N},$
listed with their multiplicities. Let $X_{0}$ be the diagonal matrix with the
$z_{j}$'s on the diagonal. Let $Y$ with diagonal entries given by%
\begin{equation}
Y_{jj}=\frac{1}{N}\sum_{k\neq j}\frac{1}{z_{j}-z_{k}} \label{Y1}%
\end{equation}
and off-diagonal entries given by%
\begin{equation}
Y_{jk}=\frac{1}{N}\frac{1}{z_{j}-z_{k}},\quad j\neq k. \label{Y2}%
\end{equation}
(Note that (\ref{Y1}) has a sum but (\ref{Y2}) does not.)

\begin{proposition}
\label{CM.prop}For any polynomial $p_{0}^{N}$ of degree $N,$ let $p_{\tau}%
^{N}$ be the heat-evolved polynomial, as in (\ref{pTauDef}). Then for all
$\tau\in\mathbb{C},$ the roots of $p_{\tau}^{N}$ are the eigenvalues of the
matrix%
\[
X_{0}-\tau Y.
\]
Specifically, if $p_{0}^{N}$ is assumed to be monic---so that $p_{\tau}^{N}$
is also monic---then the characteristic polynomial of $X_{0}-\tau Y$ is equal
to $p_{\tau}^{N}.$
\end{proposition}

\begin{proof}
The claimed result is a direct application of the formulas for the
Calogero--Moser system (\ref{secondDeriv}), with the initial velocities given
by the first-order system in (\ref{theODE}). See, for example, Section 2.7,
just above the remark on p. 18, of the monograph \cite{Etingof} by Etingof
(with a different scaling).
\end{proof}

We now discuss how Proposition \ref{CM.prop} could lead to the expected
large-$N$ behavior of the roots of $p_{\tau}^{N}.$ We would expect from
(\ref{Y1}) that the diagonal entries $Y_{jj}$ of $Y$ can be approximated for
large $N$ by the initial Cauchy transform $m_{0}$:%
\begin{equation}
Y_{jj}\approx m_{0}(z_{j}). \label{Yjj}%
\end{equation}
In particular, the diagonal entries of $Y$---and therefore of $X_{0}-\tau
Y$---should be of order 1 as $N$ tends to infinity. By contrast, the
off-diagonal entries of $Y$---and therefore the off-diagonal entries of
$X_{0}-\tau Y$---should be small. Indeed, these entries are typically of order
$1/N,$ with some small fraction of them being larger (order $1/\sqrt{N}$ if
the nearest-neighbor spacings are of order $1/\sqrt{N}$).

We then might plausibly hope that the eigenvalues of $X_{0}-\tau Y,$
especially if $\left\vert \tau\right\vert $ is small, might be well
approximated simply by the diagonal entries of $X_{0}-\tau Y.$ In that case,
Proposition \ref{CM.prop}, together with (\ref{Yjj}), would tell us that%
\[
z_{j}(\tau)\approx z_{j}-\tau m_{0}(z_{j}),
\]
which is consistent with Idea \ref{AdditiveMotion.idea} and with Point
\ref{push.point} of Conjecture \ref{introConj.conj}. We emphasize, however,
that we do not currently know how to make this line of reasoning rigorous. In
particular, the simplest estimate for the eigenvalues of almost-diagonal
matrices, namely the Gershgorin circle theorem (e.g., \cite[Theorem
6.1.1]{HornJohnson}), is not sharp enough to be useful here. Specifically, the
estimate in Gershgorin's theorem involves the sum of the absolute values of
the off-diagonal entries of the matrix over the rows of the matrix. Such sums
are of at least of order 1 in our setting.

Nevertheless, Proposition \ref{CM.prop} gives an intriguing alternative way of
attacking the problem. If nothing else, the proposition seems to give a
numerically efficient way to compute the roots of the heat-evolved polynomial
$p_{\tau}^{N}$---for one fixed $\tau$---at least in the regime where
Conjecture \ref{introConj.conj} holds. Proposition \ref{CM.prop} can, for
example, accurately replicate Figure \ref{evolvess.fig} in the
\textquotedblleft semicircular to circular\textquotedblright\ case but with
much less computational effort than numerically solving the system of ODEs in
(\ref{theODE}).

\subsection{Comparison to first-order Coulomb systems\label{coulomb.sec}}

It is instructive to compare the system of equations in (\ref{theODE}) to
first-order Coulomb systems in two dimensions, such as Eqs. (1.9) or (1.10) in
the ICM lecture of Serfaty \cite{Serf}. These equations read as, respectively,%
\begin{equation}
\frac{dz_{j}}{d\tau}=\frac{1}{N}\sum_{k\neq j}\frac{1}{\bar{z}_{j}-\bar{z}%
_{k}}, \label{coulomb1}%
\end{equation}
which is the ordinary gradient flow for the Coulomb energy of charged
particles in the plane, and%
\begin{equation}
\frac{dz_{j}}{d\tau}=\frac{i}{N}\sum_{k\neq j}\frac{1}{\bar{z}_{j}-\bar{z}%
_{k}}, \label{coulomb2}%
\end{equation}
which is the rotated gradient flow for the Coulomb energy. (One can also study
the \textit{second}-order Newtonian dynamics of charged particles in the
plane, as in Eq. (1.11) in \cite{Serf}.) Equation (\ref{coulomb1}) also
describes the dynamics of Ginzburg--Landau vortices, as in Eqs. (2.46) and
(2.47) in \cite{E}. Equations (\ref{coulomb1}) and (\ref{coulomb2}) are very
similar to (\ref{theODE}), except for the conjugates on $z_{j}$ and $z_{k}$ on
the right-hand sides.

One might then suppose that the methods used to rigorously analyze the
large-$N$ behavior of the systems (\ref{coulomb1}) and (\ref{coulomb2}%
)---described in \cite{Serf}---could be adapted to analyze (\ref{theODE}) as
well. This idea, however, turns out not to be correct. In an important
technical sense, (\ref{theODE}) is different from (\ref{coulomb1}) and
(\ref{coulomb2}), namely that the system (\ref{theODE}) is attractive in the
$x$-direction and repulsive in the $y$-direction. That is, two nearby points
$z_{j}$ and $z_{k}$ will attract each other if they are arranged horizontally
from each other and will repel each other if they are arranged vertically from
each other. This behavior is to contrasted with (\ref{coulomb1}), which is
repulsive in both directions (nearby points repel each other), and with the
(\ref{coulomb2}), which is neither attractive nor repulsive (nearby points
circle around each other). The mix of attractive and repulsive behavior in
(\ref{theODE}) means that the methods used to study (\ref{coulomb1}) and
(\ref{coulomb2}) cannot be used without substantial modifications.

\section{A PDE\ perspective\label{pde.sec}}

In this section, we offer arguments from a PDE\ perspective for Conjecture
\ref{introConj.conj} and for a multiplicative analog of that conjecture.

\subsection{The main argument\label{sect.PDE.additive}}

We consider a sequence $p^{N}$ of polynomials of degree $N$ whose empirical
root distribution converges to a compactly supported probability measure
$\mu_{0}.$ As in the statement of Conjecture \ref{introConj.conj}, we further
assume that the Cauchy transform $m_{0}$ of $\mu_{0}$ (as in Definition
\ref{PotentialCauchy.def}) is globally Lipschitz. We then let $L_{m_{0}}$
denote the Lipschitz constant of $m_{0}$:
\begin{equation}
L_{m_{0}}=\sup_{z_{1}\neq z_{2}}\left\vert \frac{m_{0}(z_{1})-m_{0}(z_{2}%
)}{z_{1}-z_{2}}\right\vert <\infty. \label{Lm0}%
\end{equation}
We now let $p_{\tau}^{N}$ denote the associated heat-evolved polynomials, as
in (\ref{pTauDef}). We now offer arguments for Points 1 and 2 of Conjecture
\ref{introConj.conj} from a PDE perspective.

\subsubsection{Point 1 of Conjecture \ref{introConj.conj}}

Point 1 of the conjecture---that the log potential of the limiting root
distribution satisfies the PDE (\ref{addPDEintro})---is supported by the
following proposition.

\begin{proposition}
\label{thm.log.pot.N} The logarithmic potential $S^{N}(z,\tau)=\frac{1}{N}%
\log|p^{N}(z,\tau)|^{2}$ satisfies the following PDE
\begin{equation}
\frac{\partial}{\partial\tau}S^{N}(z,\tau)=\frac{1}{2}\left(  \frac{\partial
S^{N}}{\partial z}\right)  ^{2}+\frac{1}{2N}\frac{\partial^{2}S^{N}}{\partial
z^{2}} \label{eq.finiteN.PDE}%
\end{equation}
whenever $z$ is not a root of $p^{N}(z,\tau)$. Alternatively, the second term
in \eqref{eq.finiteN.PDE} can be written in terms of the roots $\{z_{j}%
(\tau)\}_{j=1}^{N}$ of $p^{N}(z,\tau)$ so that
\[
\frac{\partial}{\partial\tau}S^{N}(z,\tau)=\frac{1}{2}\left(  \frac{\partial
S^{N}}{\partial z}\right)  ^{2}-\frac{1}{2N^{2}}\sum_{j=1}^{N}\frac
{1}{(z-z_{j}(\tau))^{2}}.
\]

\end{proposition}

\begin{proof}
Notice that $p^{N}(z,\tau)$ is holomorphic in $\tau$ and satisfies the
equation
\[
\frac{\partial}{\partial\tau}p^{N}(z,\tau)=\frac{1}{2N}\frac{\partial^{2}%
}{\partial z^{2}}p^{N}(z,\tau).
\]
Around $(z,\tau)$ where $p^{N}(z,\tau)\neq0$, $p^{N}$ has a branch of $\log$,
so that
\[
S^{N}=\frac{1}{N}\log p^{N}+\frac{1}{N}\overline{\log p^{N}}.
\]
Since $\log p^{N}$ is holomorphic in $z$ and $\tau$, $\overline{\log p^{N}}$
is anti-holomorphic in $z$ and $\tau$. The $\tau$-derivative of $S^{N}$ is
given by
\begin{align*}
\frac{\partial S^{N}}{\partial\tau}  &  =\frac{1}{2N^{2}}\frac{(\partial
^{2}/\partial z^{2})p^{N}(z,\tau)}{p^{N}(z,\tau)}\\
&  =\frac{1}{2N^{2}}\frac{(\partial^{2}/\partial z^{2})e^{NS^{N}}}{e^{NS^{N}}%
}\\
&  =\frac{1}{2N^{ 2}}\left[  N^{2}\left(  \frac{\partial S^{N}}{\partial
z}\right)  ^{2}+N \frac{\partial^{2}S^{N}}{\partial z^{2}}\right]  ,
\end{align*}
which simplifies to \eqref{eq.finiteN.PDE}.

Finally, by writing $p^{N}(z,\tau)$ into a product of linear factors, the
first and second $z$-derivatives of $S^{N}$ are given by
\[
\frac{\partial S^{N}}{\partial z}=\frac{1}{N}\frac{\partial/\partial
zp^{N}(z,\tau)}{p^{N}(z,\tau)}=\frac{1}{N}\sum_{j=1}^{N}\frac{1}{z-z_{j}%
(\tau)}%
\]
and
\[
\frac{\partial^{2}S^{N}}{\partial z^{2}}=-\frac{1}{N}\sum_{j=1}^{N}\frac
{1}{(z-z_{j}(\tau))^{2}}.
\]
This proves the last assertion.
\end{proof}

The PDE in \eqref{eq.finiteN.PDE} formally converges to the desired PDE
(\ref{addPDEintro}) as $N\rightarrow\infty$. We now look more carefully at the
question of whether the error term (the second term on the right-hand side of
\eqref{eq.finiteN.PDE}) is actually small for large $N$.

Recall that
\[
\frac{\partial^{2}S^{N}}{\partial z^{2}}=-\frac{1}{N}\sum_{i=1}^{N}\frac
{1}{(z-z_{i}(\tau))^{2}}.
\]
We expect the estimate
\[
\frac{\partial^{2}S^{N}}{\partial z^{2}}=-\frac{1}{N}\sum_{i=1}^{N}\frac
{1}{(z-z_{i}(\tau))^{2}}\approx-\text{P.V.}\int\frac{1}{(z-w)^{2}}d\mu_{\tau
}(w).
\]
Thus, the principal value integral should be finite as long as $|\tau|$ is
small enough so that $\mu_{\tau}$ remains absolutely continuous on
$\mathbb{C}$. If this is true, the term $\frac{1}{2N}\frac{\partial^{2}S^{N}%
}{\partial z^{2}}\rightarrow0$ as $N\rightarrow\infty$, and the limit of
$S^{N}$ in Theorem~\ref{thm.log.pot.N} becomes the solution of the desired PDE
(\ref{addPDEintro}).

Furthermore, the argument in the \textquotedblleft Details of Step
1\textquotedblright\ portion of Section \ref{details.sec} (based on Theorem
1.1 and Remark 1.1 in \cite{Li}) indicates that if the solution is $C^{1},$ it
will automatically by $C^{1,1},$ as claimed in Point 1 of Conjecture
\ref{introConj.conj}.

\subsubsection{Brief argument for Point 2 of Conjecture \ref{introConj.conj}%
\label{brief.sec}\label{sect.pushforward}}

In this section, we give a brief argument in three steps why we believe the
push-foward relation in Point 2 of Conjecture \ref{introConj.conj}. Briefly,
the argument is that \textit{Point 2 should follow from Point 1}, modulo
technical issues having to do with the regularity of the log potential. A
separate argument for Point 2 of the conjecture was provided in Section
\ref{dynamical.sec}; see Idea \ref{AdditiveMotion.idea}.

\textbf{Step 1:} We expect the PDE (\ref{addPDEintro}) has a unique solution
for small time; that is, the logarithmic potential $S$ of the limiting root
distribution in Conjecture \ref{introConj.conj} is the \textit{unique}
solution of the PDE (\ref{addPDEintro}) for small $|\tau|$. The PDE
(\ref{addPDEintro}) is of Hamilton--Jacobi type and (sufficiently regular)
solutions can be analyzed using the method of characteristics. The global
Lipschitz assumption on the Cauchy transform $m_{0}$ of $\mu_{0}$ will
guarantee that, for small enough $\tau$, each point in the plane can be
reached by a unique characteristic curve, giving uniqueness.

\textbf{Step 2:} The density $W$ of the limiting measure $\mu_{\tau}$ can be
calculated from the logarithmic potential $S$ by
\[
W=\frac{1}{\pi}\frac{\partial^{2}S}{\partial z\partial\bar{z}}.
\]
If $S$ is smooth enough, we may apply $\frac{1}{\pi}\frac{\partial}%
{\partial\bar{z}}$ and then $\frac{\partial}{\partial z}$ on both sides of the
PDE (\ref{addPDEintro}), giving
\[
\frac{\partial}{\partial\tau}\left(  \frac{1}{\pi}\frac{\partial^{2}%
S}{\partial z\partial\bar{z}}\right)  =\frac{\partial}{\partial z}\left[
\left(  \frac{\partial S}{\partial z}\right)  \cdot\frac{1}{\pi}\left(
\frac{\partial^{2}S}{\partial\bar{z}\partial z}\right)  \right]  .
\]
Hence, the density $W$ satisfies a (complex-time) continuity equation
\begin{equation}
\frac{\partial W}{\partial\tau}+\frac{\partial}{\partial z}\left[
-\frac{\partial S}{\partial z}W\right]  =0, \label{eq.continuity}%
\end{equation}
where we treat $S$ as a known quantity---the unique solution of the PDE
(\ref{addPDEintro}).

We can use general results of continuity equations by converting the
complex-time continuity equation into a real-time continuity equation.
Equation \eqref{eq.continuity} can be interpreted as saying that the measure
$\mu_{\tau}=W\,dx\,dy$ flows along the vector field $-\partial S/\partial z$.
(Here we interpret the complex number $-\partial S/\partial z$ as a vector in
the plane.) As long as $\partial S/\partial z$ is Lipschitz in $z$, general
results for continuity equations (e.g., \cite{AmbrosioCrippa}) will tell us
that the measure $\mu_{\tau}$ is the push-forward of $\mu_{0}$ by the map
obtained by flowing along $-\partial S/\partial z$.

\textbf{Step 3:} Step 2 in the argument says that $\mu_{\tau}$ is the
push-forward of $\mu_{0}$ under the flow of $-\partial S/\partial z$. We now
show that this map can be computed in a simple, explicit form as in Point 2 of
Conjecture \ref{introConj.conj}.

The Hamilton--Jacobi equation (\ref{addPDEintro}) can be analyzed using the
Hamilton's equations associated to the (complex-variable) Hamiltonian
\begin{equation}
H(z,p)=-\frac{1}{2}p^{2}, \label{eq.complex.Hamiltonian}%
\end{equation}
where $p$ is the (complex-valued) momentum variable. The Hamiltonian is
holomorphic in $z$ and $p$, and independent of $z$. While we are not aware of
a general theory on complex-time Hamilton--Jacobi equations, we can do an
analysis similar to \cite[Section 5]{HHmult} by converting the equation into a
real-time Hamilton--Jacobi equation. As we will see in Section
\ref{appendix.HJ}, the result will be Hamilton--Jacobi formulas similar to the
usual real-time formulas, as in the next paragraph.

The Hamilton's equations associated to the Hamiltonian
\eqref{eq.complex.Hamiltonian} are
\[
\frac{dz}{d\tau}=\frac{\partial H}{\partial p}=-p;\quad\frac{dp}{d\tau}%
=-\frac{\partial H}{\partial z}=0,
\]
where we look for solutions that are holomorphic functions of $\tau$. The
initial condition $z_{0}$ for $z(\tau)$ can be chosen arbitrarily on
$\mathbb{C}$, but the initial condition $p_{0}$ for $p(\tau)$ is chosen to be
\begin{equation}
p_{0}=\frac{\partial}{\partial z_{0}}S(z_{0},0)=m_{0}(z_{0}). \label{p0def}%
\end{equation}
The solution of the Hamilton's equations are
\begin{equation}
p(\tau)=p_{0};\quad z(\tau)=z_{0}-\tau p_{0}. \label{eq.char.curves}%
\end{equation}
The solutions \eqref{eq.char.curves} are called the \textbf{characteristic
curves} of the PDE (\ref{addPDEintro}). The first and second Hamilton--Jacobi
formulas then hold:
\begin{align}
S(z(\tau),\tau)  &  =S(z_{0},0)-\mathrm{Re}[\tau p_{0}^{2}]\label{eq.1st.HJ}\\
\frac{\partial S}{\partial z}(z(\tau),\tau)  &  =p(\tau). \label{eq.2nd.HJ}%
\end{align}
These two formulas are derived in Section \ref{appendix.HJ}, Proposition
\ref{prop.HJ.formulas}.

If we combine the Hamilton's equations and the second Hamilton--Jacobi
formula, we have
\[
\frac{dz}{d\tau}=-p(\tau)=-\frac{\partial S}{\partial z}(z(\tau),\tau).
\]
The right-hand side of the above equation is exactly the \textbf{vector field
occurring in} the continuity equation \eqref{eq.continuity}. Thus, the
integral curves for the vector field $-\partial S/\partial z$ are exactly the
characteristic curves of the Hamilton--Jacobi equation (\ref{addPDEintro}) .
It is notable that the integral curves are holomorphic in $\tau$.

The integral curves of $-\partial S/\partial z$ are the characteristic curves
$z_{0}-\tau p_{0}$, and $p_{0}$ is $m_{0}(z_{0})$; the map $T_{\tau}$ defined
in Point 2 of Conjecture \ref{introConj.conj} is just the characteristic
curves evaluated at $\tau$. Recall that the general results of the continuity
equation in Step 2 tells us that the density $W$ of the limiting root
distribution flows along the characteristic curves. This gives an argument for
Point 2 of Conjecture \ref{introConj.conj}.

\subsubsection{Details of Point 2 of Conjecture \ref{introConj.conj}%
\label{details.sec}}

In this section, we fill in the missing details in Section \ref{brief.sec}. We
also note certain places where the argument is not rigorous, for reasons
having to do with the regularity of the log potential $S.$

\textbf{Details of Step 1: }The PDE (\ref{addPDEintro}) is of Hamilton--Jacobi
type with complex time and space variables. In Section \ref{appendix.HJ}, we
transform the PDE into another Hamilton--Jacobi type PDE with a real time and
two real space variables. Now, under quite general conditions (e.g.,
\cite{CL}), Hamilton--Jacobi equations have unique viscosity solutions.
Furthermore, any $C^{1}$ solution is automatically a viscosity solution, so
there can be at most one $C^{1}$ solution, giving the desired uniqueness result.

We also note that by Theorem 1.1 and Remark 1.1 in \cite{Li}, one can
construct a $C^{1,1}$ solution to a Hamilton--Jacobi equation by the method of
characteristics, \textit{provided} that the initial data are $C^{1,1}$ and
that the flow along the characteristic curves is a bi-Lipschitz homeomorphism.
Thus, under these assumptions, any $C^{1}$ solution must be given by the
method of characteristics---and this $C^{1}$ solution will actually be
$C^{1,1}.$

We now show that, in fact, every point in the plane is hit by a unique
characteristic curve, provided that $\tau$ is small enough. We first note
that, in light of (\ref{p0def}) and (\ref{eq.char.curves}), the characteristic
curves can be computed in terms of the transport map as
\[
z(\tau)=T_{\tau}(z_{0}).
\]
We then establish nice behavior for $T_{\tau}$, for small $\tau.$

\begin{lemma}
\label{lem.Ttau.prop} For all $|\tau|<L_{m_{0}}^{-1}$ (where $L_{m_{0}}$ is
the Lipschitz constant of $m_{0}$), $T_{\tau}$ is a bi-Lipschitz homeomorphism
of $\mathbb{C}$ to $\mathbb{C}.$ Indeed, $T_{\tau}$ is uniformly bi-Lipschitz
for $\left\vert \tau\right\vert \leq c,$ for any $c<L_{m_{0}}.$
\end{lemma}

\begin{proof}
Suppose that $T_{\tau}(z_{1})-T_{\tau}(z_{2})=0$ for some $z_{1}\neq z_{2}$.
Then after an algebraic calculation, we get
\[
\frac{m_{0}(z_{1})-m_{0}(z_{2})}{z_{1}-z_{2}}=-\frac{1}{\tau},
\]
which is impossible if $\left\vert \tau\right\vert <L_{m_{0}}^{-1}.$ This
proves injectivity of $T_{\tau}.$ To prove surjectivity of $T_{\tau},$ we note
that $T_{\tau}(z)\approx z$ for large $|z|$. Thus, we can extend $T_{\tau}$ to
a continuous map on $\mathbb{S}^{2}$ by defining $T_{\tau}(\infty)=\infty$.
Suppose now that, on the contrary, there exists $z_{0}\in\mathbb{C}$ such that
the image $T_{\tau}(\mathbb{S}^{2})$ does not contain $z_{0}$. Then this
extended $T_{\tau}$ is a continuous map from $\mathbb{S}^{2}$ into
$\mathbb{S}\setminus\{z_{0}\}\cong\mathbb{R}^{2}$. By the Borsuk--Ulam
theorem, $T_{\tau}$ cannot be a injective map, contradicting the previously
established injectivity of $T_{\tau}.$ (Alternatively, surjectivity may be
proved using the fundamental group of the punctured plane, following a
well-known proof of the fundamental theorem of algebra. See, for example, the
proof of Theorem 56.1 in \cite{Munkres}.)

Finally, we address the claimed bi-Lipschitz property of $T_{\tau}.$ We
compute that
\[
\frac{T_{\tau}(z_{1})-T_{\tau}(z_{2})}{z_{1}-z_{2}}=1+\tau\frac{m_{0}%
(z_{1})-m_{0}(z_{2})}{z_{1}-z_{2}}.
\]
Thus, for any $c<L_{m_{0}}^{-1}$,
\[
0<1-cL_{m_{0}}\leq\left\vert \frac{T_{\tau}(z_{1})-T_{\tau}(z_{2})}%
{z_{1}-z_{2}}\right\vert \leq1+cL_{m_{0}}%
\]
if $|\tau|\leq c$.
\end{proof}

Before we move to the details of Step 2, we first define the time $\tau_{\max
}$ where the PDE and the transport map $T_{\tau}$ have nice behaviors.

\begin{notation}
Let $\tau_{\max}=\min\left(  C,L_{m_{0}}^{-1}\right)  $, where $C$ is the
constant in Conjecture \ref{introConj.conj} and $L_{m_{0}}$ is the Lipschitz
constant of $m_{0}$, as in (\ref{Lm0}).
\end{notation}

\textbf{Details of Step 2:} Step 2 of the argument in Section
\ref{sect.pushforward} concerns the deformation of the measure $\mu_{\tau}$.
Let $W(\cdot,\tau)$ be the density of $\mu_{\tau}$. If $S$ is $C^{3}$, then we
can take $\frac{1}{\pi}\partial^{2}/\partial z\partial\bar{z}$ on both sides
of (\ref{addPDEintro}) to get the continuity equation . We expect that if the
solution is only $C^{1,1},$ the resulting continuity equation will still hold
in the weak sense.

In the following result, which is proven in Section \ref{appendix.HJ}, we
transform the complex-time continuity equation \eqref{eq.continuity} into a
(real-time) continuity equation for $W$. Then in Step 3, we derive that
$\mu_{\tau}$ is the push-forward measure of $\mu_{0}$ by the map $T_{\tau}$.

\begin{proposition}
\label{prop.real.cont.eq} Suppose that $W$ satisfies the complex-time
continuity equation \eqref{eq.continuity}. Let
\[
W^{\tau}(z,t)=W(z,t\tau).
\]
Then $W^{\tau}$ satisfies the real-time continuity equation
\begin{equation}
\frac{\partial W^{\tau}}{\partial t}+\nabla\cdot(bW^{\tau})=0
\label{eq.cont.equation}%
\end{equation}
where $b$ is the vector field
\begin{equation}
b(z,\tau)=b(z,\tau)=-%
\begin{pmatrix}
\tau_{1} & -\tau_{2}\\
\tau_{2} & \tau_{1}%
\end{pmatrix}%
\begin{pmatrix}
\frac{1}{2}\frac{\partial S}{\partial x}\\
{-}\frac{1}{2}\frac{\partial S}{\partial y}%
\end{pmatrix}
=-\operatorname{vect}\left(  \tau\frac{\partial S}{\partial z}(z,\tau)\right)
\label{eq.vectorfield.appendix}%
\end{equation}
where the operator $\operatorname{vect}$ takes a complex number $x+iy$ to the
vector $(x,y)$ in $\mathbb{R}^{2}$.
\end{proposition}

By general results of the continuity equation (see, for example, Proposition
3.1 and Theorem 4.1 of \cite{AmbrosioCrippa}), if the initial condition of the
continuity equation is the density of a measure, then the solution of a
continuity equation is given by the push-forward of the measure by the flow of
the vector field. We will show in details of Step 3 that the ODE describing
the flow along the vector field (\ref{eq.vectorfield.appendix}) has a unique
solution for every initial point on $\mathbb{C}$. Theorem 4.1 of
\cite{AmbrosioCrippa} then guarantees us the solution of the continuity
equation is unique.

\begin{proof}
[Proof of Proposition \ref{prop.real.cont.eq}]Using the chain rule, we
compute
\[
\frac{\partial W^{\tau}}{\partial t}=2\operatorname{Re}\left[  \tau
\frac{\partial W}{\partial\tau}\right]  =\operatorname{Re}\left[
2\frac{\partial}{\partial z}\left(  \tau\frac{\partial S}{\partial z}W^{\tau
}\right)  \right]  .
\]
Expanding the right-hand side, we get
\[
\frac{\partial W^{\tau}}{\partial t}=\frac{\partial}{\partial x}\left[
\left(  \frac{\tau_{1}}{2}\frac{\partial S}{\partial x}+\frac{\tau_{2}}%
{2}\frac{\partial S}{\partial y}\right)  W^{\tau}\right]  +\frac{\partial
}{\partial y}\left[  \left(  \frac{\tau_{2}}{2}\frac{\partial S}{\partial
x}-\frac{\tau_{1}}{2}\frac{\partial S}{\partial y}\right)  W^{\tau}\right]  .
\]
This means we can write the above equation into the standard form of a
continuity equation
\[
\frac{\partial W^{\tau}}{\partial t}+\nabla\cdot(bW^{\tau})=0,
\]
where $b$ is the vector field in the statement of the proposition.
\end{proof}

\textbf{Details of Step 3:} We consider the real-time continuity equation in
(\ref{eq.cont.equation}) and we consider the integral curves of the associated
vector field $-\mathrm{vect}(\tau\partial S/\partial z).$ That is, we look for
solutions $w_{\tau}(t)$ to
\begin{equation}
\frac{dw_{\tau}(t)}{dt}=-\tau\frac{\partial S}{\partial z}(w_{\tau}(t),t\tau).
\label{ODEforw}%
\end{equation}
Now, assuming that the Hamilton--Jacobi formulas are applicable to the
(real-variables version of) the equation for $S,$ Proposition
\ref{prop.Lipschitz.dSdz} will tell us that the right-hand side of
(\ref{ODEforw}) will be Lipschitz. Thus, by Picard--Lindel\"{o}f theorem, the
ODE has a unique solutions. Then by Theorems 3.1 and 4.1 of
\cite{AmbrosioCrippa}, the continuity equation (\ref{eq.cont.equation}) has a
unique solution, given by push-forward along the curves in (\ref{ODEforw}).

We then solve (\ref{ODEforw}) by a real-variables version of the argument we
gave for Step 3 in Section \ref{brief.sec}. Specifically, we consider the
function $S^{\tau}(z,t)=S(z,t\tau).$ The proof of Proposition
\ref{prop.HJ.formulas} in Section \ref{appendix.HJ} will show that the curves
$z^{\tau}(t)=z_{0}-t\tau p_{0}$ and $p^{\tau}(t)=p_{0}$ satisfy%
\[
\frac{dz^{\tau}}{dt}=-\tau p^{\tau}(t)=-\tau\frac{\partial S^{\tau}}{\partial
z}(z^{\tau}(t),t)=-\tau\frac{\partial S}{\partial z}(z^{\tau}(t),t\tau).
\]
Thus, the curves $t\mapsto z^{\tau}(t)$ are the unique solutions of
(\ref{ODEforw}) with the initial condition $z_{0}.$ Taking $t=1,$ we conclude
that the measure $W(z,\tau)~dx~dy$ is the push-forward of $W(z,0)~dx~dy$ under
the map $z\mapsto z-\tau p_{0}=z-\tau m_{0}(z).$

\subsection{Auxiliary calculations\label{appendix.HJ}}

We will use the ordinary Hamilton--Jacobi method to prove that the solution of
the PDE (\ref{addPDEintro}) satisfies the Hamilton--Jacobi formulas
(\ref{eq.1st.HJ}) and (\ref{eq.2nd.HJ}). The proof requires the solution $S$
to be $C^{2},$ although we expect that the result will continue to hold if the
solution is only $C^{1,1}.$

\begin{proposition}
\label{prop.HJ.formulas} Suppose $S$ is a $C^{2}$ solution of
(\ref{addPDEintro}) for $\left\vert \tau\right\vert \leq C.$ Then the values
of $S$ and $\partial S/\partial z$ can be calculated using the characteristic
curves \eqref{eq.char.curves} as in (\ref{eq.1st.HJ}) and (\ref{eq.2nd.HJ}).
\end{proposition}

\begin{proof}
We first turn the PDE (\ref{addPDEintro}) into a PDE with real time, and apply
the first Hamilton--Jacobi formulas. (See, for example, Proposition 5.3 of
\cite{DHKBrown}.) For each $\left\vert \tau\right\vert \leq C$, define
\[
S^{\tau}(z,t)=S(z,t\tau).
\]
The $z$-derivative of $S^{\tau}$ and $S$ coincides with the same $(z,\tau)$,
so that
\begin{align*}
\frac{\partial S^{\tau}}{\partial t}  &  =\frac{\partial S}{\partial\tau}%
\frac{\partial t\tau}{\partial t}+\frac{\partial S}{\partial\bar{\tau}}%
\frac{\partial t\bar{\tau}}{\partial t}\\
&  =2\mathrm{Re}\left[  \frac{\tau}{2}\left(  \frac{\partial S^{\tau}%
}{\partial z}\right)  ^{2}\right] \\
&  =\frac{\operatorname{Re}\tau}{4}\left[  \left(  \frac{\partial S^{\tau}%
}{\partial x}\right)  ^{2}-\left(  \frac{\partial S^{\tau}}{\partial
y}\right)  ^{2}\right]  +\frac{\operatorname{Im}\tau}{2}\frac{\partial
S^{\tau}}{\partial x}\frac{\partial S^{\tau}}{\partial y},
\end{align*}
where we have written $z$ into real variables by $z=x+iy$. Thus, $S^{\tau
}(z,t)$ satisfies a Hamilton--Jacobi PDE, with Hamiltonian
\[
H^{\tau}(x,y,p_{x},p_{y})=-\frac{1}{4}\mathrm{Re}[\tau(p_{x}-ip_{y})^{2}].
\]

The Hamilton's equations for this Hamiltonian are given as
\[
\dot{x}=\frac{\partial H^{\tau}}{\partial p_{x}};\;\dot{y}=\frac{\partial
H^{\tau}}{\partial p_{y}};\;\dot{p}_{x}=-\frac{\partial H^{\tau}}{\partial
x};\;\dot{p}_{y}=-\frac{\partial H^{\tau}}{\partial y}.
\]
Once the initial condition $z_{0}=x_{0}+iy_{0}$ is chosen, the initial
conditions $p_{x,0}$ and $p_{y,0}$ of the Hamilton's equations are taken to
be
\[
p_{x,0}=\frac{\partial S^{\tau}}{\partial x}(z_{0},0);\quad p_{y,0}%
=\frac{\partial S^{\tau}}{\partial y}(z_{0},0).
\]
We then easily obtain
\begin{align*}
&  p_{x}(t)=p_{x,0};\quad p_{y}(t)=p_{y,0}\\
&  x(t)=x_{0}-\frac{t}{2}\mathrm{Re}[\tau(p_{x,0}-ip_{y,0})]\\
&  y(t)=y_{0}-\frac{t}{2}\mathrm{Im}[\tau(p_{x,0}-ip_{y,0})].
\end{align*}

We then rewrite our results using the complex-valued functions $z^{\tau}$ and
$p^{\tau}$ given by
\begin{align*}
z^{\tau}(t)  &  =x(t)+iy(t)=z_{0}-t\tau p_{0}\\
p^{\tau}(t)  &  =\frac{1}{2}(p_{x}(t)-ip_{y}(t)).
\end{align*}
We can now apply Eqs. (5.11) and (5.12) of \cite{DHKBrown} to get the
formulas
\begin{align}
S^{\tau}(z^{\tau}(t),t)  &  =S^{\tau}(z_{0},0)-\mathrm{Re}\left[  \tau\left(
\frac{1}{2}(p_{x_{0}}-ip_{y,0})\right)  ^{2}\right]  t\nonumber\\
&  =S^{\tau}(z_{0},0)-t\mathrm{Re}[\tau p_{0}^{2}] \label{1HJ}%
\end{align}
and%
\begin{equation}
\frac{\partial S^{\tau}}{\partial z}(z^{\tau}(t),t)=p^{\tau}(t). \label{2HJ}%
\end{equation}
Evaluating (\ref{1HJ}) and (\ref{2HJ}) at $t=1$ gives the claimed formulas.
\end{proof}

\begin{proposition}
\label{prop.Lipschitz.dSdz} Assume that the conclusion of Proposition
\ref{prop.HJ.formulas} holds; that is, the characteristic curves
\eqref{eq.char.curves} describe the values of $S$ and $\partial S/\partial z$
by the formulas \eqref{eq.1st.HJ} and \eqref{eq.2nd.HJ}. Then given any
$c<L_{m_{0}}^{-1}$, for any $\left\vert \tau\right\vert \leq c$, $\partial
S/\partial z(\cdot,\tau)$ is uniformly Lipschitz.
\end{proposition}

\begin{proof}
By the second Hamilton--Jacobi formula \eqref{eq.2nd.HJ}, the global Lipschitz
assumption on $m_{0}$, and Lemma \ref{lem.Ttau.prop}, we have
\begin{align*}
\left\vert \frac{\partial S}{\partial z}(z_{1},\tau)-\frac{\partial
S}{\partial z}(z_{1},\tau)\right\vert  &  =|m_{0}(T_{\tau}^{-1}(z_{1}%
))-m_{0}(T_{\tau}^{-1}(z_{2}))|\\
&  \leq L_{m_{0}}|T_{\tau}^{-1}(z_{1})-T_{\tau}^{-1}(z_{2})|\\
&  \leq L_{m_{0}}(1+cL_{m_{0}})\left\vert z_{1}-z_{2}\right\vert
\end{align*}
for all $\left\vert \tau\right\vert \leq c<L_{m_{0}}^{-1}$.
\end{proof}

\section{Evolution of the holomorphic moments\label{moments.sec}}

If $\mu$ is a compactly supported probability measure on the plane, we let
$M_{a}$ be the $a$-th \textbf{holomorphic moment} of $\mu,$ given as%
\[
M_{a}=\int_{\mathbb{C}}z^{a}~d\mu(z),\quad a=0,1,2,\ldots,
\]
so that $M_{0}=1.$ We emphasize that these moments do not uniquely determine
the measure. (Every rotationally invariant choice of $\mu,$ for example, gives
$M_{a}=0$ for all $a\geq1.$) Nevertheless, the holomorphic moments contain a
lot of information about the measure.

Now let $p_{0}^{N}$ be a polynomial of degree $N$ and let $p_{\tau}^{N}$ be
the heat evolution of $p_{0}^{N}$, as in (\ref{heatDef}). We then let
$M_{a,N}(\tau)$ be the $a$-th holomorphic moment of the empirical root measure
of $p_{\tau}^{N},$ that is,%
\begin{equation}
M_{a,N}(\tau)=\frac{1}{N}\sum_{j=1}^{N}(z_{j,N}(\tau))^{a}, \label{Man}%
\end{equation}
where $\{z_{j,N}(\tau)\}_{j=1}^{N}$ are the roots of $p_{\tau}^{N}.$

The following theorem gives mechanism for computing these moments inductively
as a function of $a=0,1,\ldots$, allows us to compute the large-$N$ limit of
the moments, and shows that the large-$N$ moments are consistent with the
prediction of the PDE\ in Point \ref{pde.point} in Conjecture
\ref{introConj.conj}. We may therefore say that Point \ref{pde.point} of
Conjecture \ref{introConj.conj} holds \textquotedblleft at the level of the
holomorphic moments.\textquotedblright\ 

\begin{theorem}
\label{holoMoment.thm}We have the following results.

\begin{enumerate}
\item For any polynomial $p_{0}^{N}$ of degree $N,$ the moments $M_{a,N}%
(\tau)$ in (\ref{Man}) satisfy%
\begin{equation}
\frac{dM_{0,N}}{d\tau}=\frac{dM_{1,N}}{d\tau}=0 \label{firstTwoMoments}%
\end{equation}
and for $a\geq2,$%
\begin{equation}
\frac{dM_{a,N}}{d\tau}=\frac{1}{2N}a(a-1)M_{a-2,N}(\tau)-\frac{a}{2}\sum
_{b=0}^{a-2}M_{b,N}(\tau)M_{a-b-2,N}(\tau). \label{dMdTauN}%
\end{equation}

\item Suppose now that $p_{0}^{N}$ is a sequence of polynomials of degree $N$
such that the limit%
\[
M_{a}(0):=\lim_{N\rightarrow\infty}M_{a,N}(0)
\]
exists for all $a.$ Then the limit%
\[
M_{a}(\tau):=\lim_{N\rightarrow\infty}M_{a,N}(\tau)
\]
exists for all $\tau$ and these limiting moments satisfy the limiting version
of (\ref{dMdTauN}), namely%
\begin{equation}
\frac{dM_{a}}{d\tau}=-\frac{a}{2}\sum_{b=0}^{a-2}M_{b}(\tau)M_{a-b-2}(\tau).
\label{dMdTau}%
\end{equation}

\item Suppose $\mu_{\tau}$ is a family of compactly supported probability
measures whose log potential $S(z,\tau)$ satisfies the PDE (\ref{addPDEintro})
for $z$ near infinity and $\tau$ in some open set $U\subset\mathbb{C}.$ Let
$\hat{M}_{a}(\tau)$ be the $a$-th holomorphic moment of $\mu_{\tau}.$ Then
these moments also satisfy (\ref{dMdTau}):
\begin{equation}
\frac{d\hat{M}_{a}}{d\tau}=-\frac{a}{2}\sum_{b=0}^{a-2}\hat{M}_{b}(\tau
)\hat{M}_{a-b-2}(\tau) \label{Mhat}%
\end{equation}
for $\tau\in U.$
\end{enumerate}
\end{theorem}

Although we will establish (\ref{dMdTauN}) using the PDE\ methods from Section
\ref{pde.sec}, it is also possible to verify the result using the ODE system
(\ref{theODE}) in Section \ref{dynAdd.sec}.

\begin{proof}
By (\ref{eq.finiteN.PDE}), the Cauchy transform $m^{N}(z,\tau)=\frac{\partial
S^{N}}{\partial z}(z,\tau)$ of the empirical root distribution of $p_{\tau
}^{N}$ satisfies the PDE%
\begin{equation}
\frac{\partial m^{N}}{\partial\tau}=\frac{\partial}{\partial z}\left[
\frac{1}{2}(m^{N})^{2}+\frac{1}{2N}\frac{\partial m^{N}}{\partial z}\right]  .
\label{MNpde}%
\end{equation}
Meanwhile, the Cauchy transform $m$ of any compactly supported probability
measure $\mu$ can be computed near infinity as%
\begin{align*}
m(z)  &  :=\int_{\mathbb{C}}\frac{1}{z-w}~d\mu(w)\\
&  =\frac{1}{z}\int_{\mathbb{C}}\left(  1+\frac{w}{z}+\frac{w^{2}}{z^{2}%
}+\cdots\right)  ~d\mu(w)\\
&  =\frac{1}{z}\left(  1+\frac{a_{1}}{z}+\frac{a_{2}}{z^{2}}+\cdots\right)  ,
\end{align*}
where $\{a_{j}\}_{j=0}^{\infty}$ are the holomorphic moments of $\mu.$

We then have%
\begin{equation}
\frac{\partial}{\partial\tau}m^{N}(z,\tau)=\frac{1}{z}\left(  \frac
{a_{1,N}^{\prime}(\tau)}{z}+\frac{a_{2,N}^{\prime}(\tau)}{z^{2}}+\frac
{a_{3,N}^{\prime}(\tau)}{z^{3}}+\cdots\right)  \label{dMNtauLHS}%
\end{equation}
and, after a short calculation,%
\begin{align}
&  \frac{\partial}{\partial z}\left[  \frac{1}{2}(m^{N})^{2}+\frac{1}{2N}%
\frac{\partial m^{N}}{\partial z}\right] \nonumber\\
&  =\sum_{j=2}^{\infty}\frac{1}{z^{j+1}}\left(  \frac{1}{2N}j(j-1)a_{j-2,N}%
(\tau)-\frac{j}{2}\sum_{k=0}^{j-2}a_{j,N}(\tau)a_{j-k,N}(\tau)\right)  .
\label{dMNtauRHS}%
\end{align}
Equating coefficients between (\ref{dMNtauLHS}) and (\ref{dMNtauRHS}) gives
the $dM_{1,N}/d\tau=0$ (for the coefficient of $1/z^{2}$) and gives
(\ref{dMdTauN}) (for the coefficient of $1/z^{j+1},$ $j\geq2$).

Meanwhile, we can solve for $M_{a,N}(\tau)$ inductively in $a,$ simply by
integrating (\ref{firstTwoMoments}) or (\ref{dMdTauN}) with respect to $\tau,$
with a fixed value of $M_{a,N}(0).$ The resulting expression will depend
continuously on the parameter $C=1/N,$ up to $C=0.$ If we integrate, let
$C\rightarrow0,$ and then differentiate in $\tau,$ we obtain (\ref{dMdTau}).

Finally, if the log potential of $\mu_{\tau}$ satisfies (\ref{addPDEintro})
near infinity, then the Cauchy transform of $\mu_{\tau}$ satisfies the PDE
(\ref{MNpde}), but without the $1/N$ term on the right-hand side. We may
therefore repeat the above calculation without the $1/N$ term to get
(\ref{Mhat}).
\end{proof}

\section{The multiplicative case\label{multmult.sec}}

In this section, we develop a \textquotedblleft
multiplicative\textquotedblright\ version of our main heat flow conjectures,
in a way that closely parallels results for the standard heat flow, which we
refer to as the \textquotedblleft additive\textquotedblright\ case. The term
\textquotedblleft multiplicative\textquotedblright\ refers to the associated
random matrix model, namely the Brownian motion in the general linear group,
which can be approximated by multiplying together a large number of
independent random matrices. (See (\ref{productApprox}).)

The multiplicative version of the heat flow we consider is essentially the one
considered by Tao in his blog post \cite{Tao2}. See Remarks
\ref{usVsTao.remark} and \ref{circleroots.remark} for a precise comparison.

\subsection{The multiplicative heat flow conjecture for random
matrices\label{multiplicativeConj.sec}}

We now define a \textquotedblleft multiplicative\textquotedblright\ version
$B_{s,\tau}^{N}$ of the elliptic random matrix model $Z_{s,\tau}^{N}.$ To do
this, let $X_{r}^{N}$ and $Y_{r}^{N}$ be independent Brownian motions in the
space of $N\times N$ Hermitian matrices, normalized so that $X_{1}^{N}$ and
$Y_{1}^{N}$ are GUEs. Then, imitating (\ref{ZNdef}), we define an
\textbf{elliptic Brownian motion} by
\[
Z_{s,\tau}^{N}(r)=e^{i\theta}(aX_{r}^{N}+ibY_{r}^{N}),
\]
where the parameters $a,$ $b,$ and $\theta$ are chosen to give the desired
values of $s$ and $\tau$ in (\ref{sAndTau}) at $r=1$. Then we introduce a
Brownian motion $B_{s,\tau}^{N}(r)$ as the solution of the following
stochastic differential equation%
\begin{align}
dB_{s,\tau}^{N}(r)  &  =B_{s,\tau}^{N}(r)\left(  i~dZ_{s,\tau}^{N}(r)-\frac
{1}{2}(s-\tau)~dr\right) \label{BstauSDE1}\\
B_{s,\tau}^{N}(0)  &  =I \label{BstauSDE2}%
\end{align}
Here, the $dr$ term on the right-hand side of (\ref{BstauSDE1}) is an It\^{o}
correction. The process $B_{s,\tau}^{N}(r)$ is a left-invariant Brownian
motion living in the general linear group $GL(N;\mathbb{C}).$ We typically
take $r=1,$ since $B_{s,\tau}^{N}(r)$ has the same distribution as
$B_{rs,r\tau}^{N}(1).$ We thus use the notation%
\[
B_{s,\tau}^{N}=\left.  B_{s,\tau}^{N}(r)\right\vert _{r=1}.
\]

The law of $B_{s,\tau}^{N}$ is the measure denoted as $\mu_{s,\tau}$ in
\cite{DHKcomplex}, where it serves to define norm on the target Hilbert space
in the complex-time Segal--Bargmann transform. When $\tau=0,$ the distribution
of $B_{s,0}^{N}$ is that of a Brownian motion in the unitary group at time
$s.$ We refer to $B_{s,\tau}^{N}$ as a \textquotedblleft
multiplicative\textquotedblright\ random matrix model because the solution to
(\ref{BstauSDE1}) and (\ref{BstauSDE2}) can be approximated by multiplying
together independent random matrices:%
\begin{equation}
B_{s,\tau}^{N}\approx\prod_{j=1}^{k}\left(  I+\frac{i}{\sqrt{k}}Z_{j}-\frac
{1}{2k}(s-\tau)I\right)  \label{productApprox}%
\end{equation}
for some large positive integer $k,$ where $Z_{1},\ldots,Z_{k}$ are
independent random matrices with the same distribution as $Z_{s,\tau}^{N}.$\ 

We let $b_{s,\tau}(r)$ be the \textquotedblleft free\textquotedblright%
\ version of $B_{s,\tau}^{N}(r),$ obtained by replacing $X_{r}^{N}$ and
$Y_{r}^{N}$ by their free counterparts and then solving the free version of
(\ref{BstauSDE1}) and (\ref{BstauSDE2}). We again take $r=1$ and use the
notation $b_{s,\tau}$ for $b_{s,\tau}(1).$ When $\tau=s,$ the process
$b_{s,s}$ is Biane's free multiplicative Brownian motion (denoted $\Lambda
_{s}$ in \cite[Section 4.2.1]{BianeJFA}). When $\tau=0,$ the process $b_{s,0}$
is Biane's free unitary Brownian motion \cite[Section 2.3]{BianeFields}. When
$\tau=t$ is real, Kemp \cite{KempLargeN} shows that the large-$N$ limit of
$B_{s,t}^{N},$ in the sense of $\ast$-distribution, is $b_{s,t}.$

In the special case $\tau=s,$ Driver, Hall, and Kemp \cite{DHKBrown}, building
on results of Hall and Kemp \cite{HKsupport}, compute the Brown measure of
$b_{s,s}$ using a novel PDE\ method. Ho and Zhong \cite{HZ} then compute the
Brown measure of $ub_{s,s},$ where $u$ is a unitary element freely independent
of $b_{s,s}.$ Hall and Ho \cite{HHmult} then compute the Brown measure of
$ub_{s,\tau}$ for general values of $\tau.$ In all cases, we believe that the
Brown measure of $ub_{s,\tau}$ coincides with the large-$N$ limit of the
empirical eigenvalue distribution of $U_{0}^{N}B_{s,\tau}^{N}$, where
$U_{0}^{N}$ is independent of $B_{s,\tau}^{N}$ and the limiting eigenvalue
distribution of $U_{0}^{N}$ equals the law of $u.$

The following conjecture describes what we believe about the Brown measure of
$ab_{s,\tau},$ where $a$ is freely independent of $b_{s,\tau}$ but not
necessarily unitary.

\begin{conjecture}
Fix $s>0$ and two complex numbers $\tau_{0}$ and $\tau$ with $\left\vert
\tau_{0}-s\right\vert <s$ and $\left\vert \tau-s\right\vert \leq s.$ Let
$\mu_{s,\tau_{0}}$ be the Brown measure of $ab_{s,\tau},$ where $a$ is
invertible and freely independent of $b_{s,\tau}$ and let $m_{s,\tau_{0}}$ be
its Cauchy transform. Then $m_{s,\tau_{0}}$ is Lipschitz continuous and the
Brown measure $\mu_{s,\tau}$ of $ab_{s,\tau}$ can be computed by the model
deformation formula%
\begin{equation}
\mu_{s,\tau}=(\Phi_{s,\tau_{0},\tau})_{\ast}(\mu_{s,\tau_{0}}),
\label{PushMult}%
\end{equation}
where the map $\Psi_{s,\tau_{0},\tau}:\mathbb{C}\rightarrow\mathbb{C}$ is
defined as%
\[
\Psi_{s,\tau_{0},\tau}(z)=z\exp\left\{  \frac{\tau-\tau_{0}}{2}\left(
2zm_{s,\tau_{0}}(z)-1\right)  \right\}  .
\]
Furthermore, as $\tau$ varies over the set $\left\vert \tau-s\right\vert <s$
(with $a$ and $s$ fixed), the log potential of $S_{s,\tau}$ is a $C^{1}$
solution of the PDE%
\begin{equation}
\frac{\partial S_{s,\tau}}{\partial\tau}=-\frac{1}{2}\left(  z^{2}\left(
\frac{\partial S}{\partial z}\right)  ^{2}-z\frac{\partial S}{\partial
z}\right)  . \label{PDEmult}%
\end{equation}

\end{conjecture}

As in the additive case, the preceding conjecture is not the main focus of the
present paper, but serves to connect Conjecture \ref{mult2.conj} below to the
general multiplicative heat flow conjecture (Conjecture
\ref{introConjMult.conj}).

\begin{remark}
If $a$ is unitary, the push-forward result in (\ref{PushMult}) is known from
Theorem 8.2 and Proposition 8.4 of \cite{HHmult}. In the unitary case, the PDE
(\ref{PDEmult}) holds except possibly on the boundary of the support of
$\mu_{s,\tau}.$ (By Proposition 6.2 and Theorem 7.5 in \cite{HHmult}, we can
set $\varepsilon=0$ in \cite[Theorem 4.2]{HHmult}, except possibly on the
boundary of the support.) Proving the conjecture for general $a$ is more
difficult than in the corresponding additive case.
\end{remark}

We now present the multiplicative version of Conjecture \ref{add1.conj}.

\begin{conjecture}
\label{mult2.conj}Fix $s>0$ and complex numbers $\tau_{0}$ and $\tau$ such
that $\left\vert \tau_{0}-s\right\vert \leq s$ and $\left\vert \tau
-s\right\vert \leq s.$ Let \thinspace$B_{s,\tau_{0}}^{N}$ and $B_{s,\tau}^{N}$
be Brownian motions as in (\ref{BstauSDE1}) and (\ref{BstauSDE2}). Suppose
$A^{N}$ a random matrix, independent of $B_{s,\tau}^{N},$ that is converging
almost sure in the sense of $\ast$-distribution to an element $a$ in a tracial
von Neumann algebra. Let $p_{s,\tau_{0}}$ be the random characteristic
polynomial of $A^{N}B_{s,\tau_{0}}^{N}$ and define a new random polynomial
$q_{s,\tau_{0},\tau}$ by%
\begin{equation}
q_{s,\tau_{0},\tau}(z)=\exp\left\{  -\frac{(\tau-\tau_{0})}{2N}\left(
z^{2}\frac{\partial^{2}}{\partial z^{2}}-(N-2)z\frac{\partial}{\partial
z}-N\right)  \right\}  p_{s,\tau_{0}}(z). \label{qstauDef}%
\end{equation}
Then the empirical root measure of $q_{s,\tau_{0},\tau}$ converges weakly
almost surely to the limiting eigenvalue distribution of $A^{N}B_{s,\tau}%
^{N}.$
\end{conjecture}

See Figure \ref{quad.fig}. See also Slide 9 in the supplemental document for
an animation in the multiplicative case.
%TCIMACRO{\FRAME{ftbpFU}{2.7466in}{2.7579in}{0pt}{\Qcb{The $(s,\tau_{0})$
%eigenvalues (top), the $(s,\tau)$ eigenvalues (bottom left), and the
%$(\tau-\tau_{0})$-evolution of the $(s,\tau_{0})$ eigenvalues (bottom right),
%for the case that $A^{N}$ is unitary with eigenvalues equally distributed
%among $\pm1$ and $\pm i.$ Shown for $s=1,$ $\tau_{0}=1,$ and $\tau=1+i.$}%
%}{\Qlb{quad.fig}}{quad.pdf}{\special{ language "Scientific Word";
%type "GRAPHIC";  maintain-aspect-ratio TRUE;  display "USEDEF";
%valid_file "F";  width 2.7466in;  height 2.7579in;  depth 0pt;
%original-width 10.0699in;  original-height 10.0837in;  cropleft "0";
%filename 'quad.pdf';file-properties "XNPEU";}} }%
%BeginExpansion
\begin{figure}[ptb]%
\centering
\includegraphics[
%%natheight=10.083700in,
%%natwidth=10.069900in,
height=2.7579in,
width=2.7466in
]%
{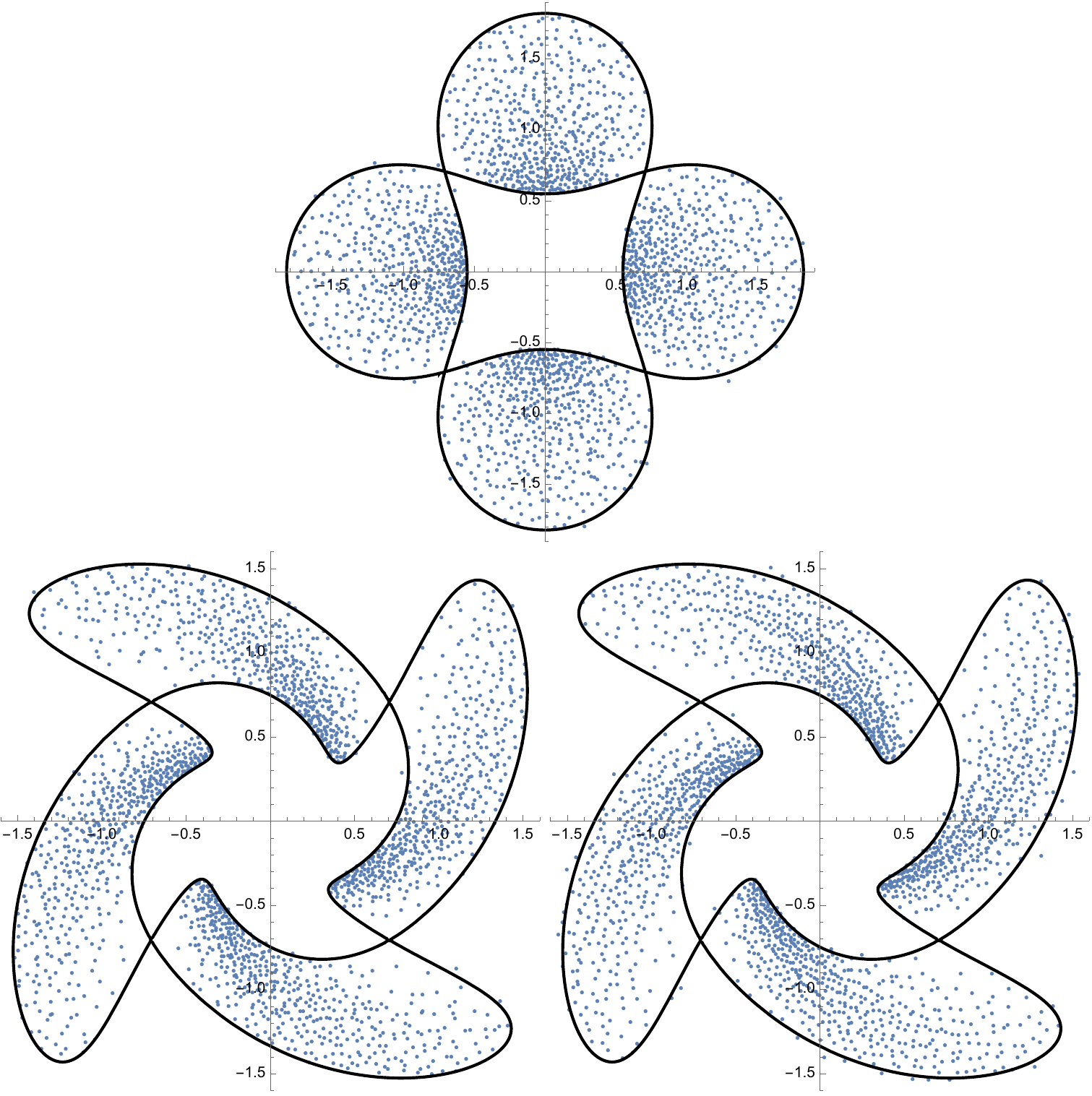}%
\caption{The $(s,\tau_{0})$ eigenvalues (top), the $(s,\tau)$ eigenvalues
(bottom left), and the $(\tau-\tau_{0})$-evolution of the $(s,\tau_{0})$
eigenvalues (bottom right), for the case that $A^{N}$ is unitary with
eigenvalues equally distributed among $\pm1$ and $\pm i.$ Shown for $s=1,$
$\tau_{0}=1,$ and $\tau=1+i.$}%
\label{quad.fig}%
\end{figure}
%EndExpansion

\begin{remark}
\label{usVsTao.remark}The right-hand side of (\ref{qstauDef}) involves the
heat flow%
\begin{equation}
\exp\left\{  -\frac{\tau}{2N}\left(  z^{2}\frac{\partial^{2}}{\partial z^{2}%
}-(N-2)z\frac{\partial}{\partial z}-N\right)  \right\}  , \label{multFlow}%
\end{equation}
evaluated with $\tau$ replaced by $\tau-\tau_{0}.$ The heat flow in
(\ref{multFlow}), meanwhile, is closely related to the one in Tao's blog post
\cite{Tao2}, with $g$ in Tao's notation identified with $N/2$ in our notation.
Suppose we evaluate the flow in (\ref{multFlow}) on a monomial and then
rescale the space variable by replacing $z$ with $e^{\frac{\tau}{2N}}z.$ After
this rescaling, we find that the monomial $z^{j}$ goes to
\[
e^{\frac{\tau N}{2}}\exp\left\{  -\frac{\tau}{2N}\left(  j^{2}-Nj\right)
\right\}  z^{j}.
\]
This agrees with the flow in \cite[Equation (2)]{Tao2}, up to a scaling of the
time variable and an overall time-dependent constant (independent of $j$)
which does not affect the roots. Note that we are scaling the space variable
by the constant $e^{\frac{\tau}{2N}},$ which approaches $1$ as $N\rightarrow
\infty.$ See also Remark \ref{circleroots.remark} in Section
\ref{multDynamics.sec}.
\end{remark}

We now define a multiplicative version $\tilde{D}$ of the second moment
function, given by
\[
\tilde{D}^{N}(s,\tau,z)=\mathbb{E}\{\left\vert \det(z-A_{0}^{N}B_{s,\tau}%
^{N})\right\vert ^{2}\}
\]

\begin{theorem}
[Deformation theorem for second moment]\label{deformationMult.thm}Suppose
$\tau_{0}$ and $\tau$ are complex numbers satisfying $\left\vert \tau
_{0}-s\right\vert \leq s$ and $\left\vert \tau-s\right\vert \leq s,$ in
accordance with (\ref{tauIneq}). Consider random matrices $A_{0}^{N}%
B_{s,\tau_{0}}^{N}$ and $A_{0}^{N}B_{s,\tau}^{N},$ respectively, where
$A_{0}^{N}$ is independent of $B_{s,\tau_{0}}^{N}$ and $B_{s,\tau}^{N}$. Then
the function $\tilde{D}^{N}$ in (\ref{DNdef}) can also be computed as%
\begin{align*}
&  \tilde{D}^{N}(s,\tau,z)\\
&  =\mathbb{E}\left\{  \left\vert \exp\left\{  -\frac{(\tau-\tau_{0})}%
{2N}\left(  z^{2}\frac{\partial^{2}}{\partial z^{2}}-(N-2)z\frac{\partial
}{\partial z}-N\right)  \right\}  \det(z-A_{0}^{N}B_{s,\tau_{0}}%
^{N})\right\vert ^{2}\right\}  .
\end{align*}

\end{theorem}

As in the additive case, we can rewrite this result as%
\[
\mathbb{E}\left\{  \left\vert \prod_{j=1}^{N}(z-z_{j}^{s,\tau})\right\vert
^{2}\right\}  =\mathbb{E}\left\{  \left\vert \prod_{j=1}^{N}(z-z_{j}%
^{s,\tau_{0}}(\tau))\right\vert ^{2}\right\}  ,
\]
where $\{z_{j}^{s,\tau}\}_{j=1}^{N}$ are the roots of the characteristic
polynomial of $A_{0}^{N}B_{s,\tau}^{N}$ and $\{z_{j}^{s,\tau_{0}}%
(\tau)\}_{j=1}^{N}$ are the roots of the characteristic polynomial of
$A_{0}^{N}B_{s,\tau_{0}}^{N},$ after evolving for time $\tau-\tau_{0}$ by the
multiplicative heat flow.

The proof of Theorem \ref{deformationMult.thm} is similar to the proof of
Theorem \ref{deformationAdditive.thm} in the additive case.

We briefly indicate the differences. We construct elements $X_{j}$ and $Y_{j}$
as in the additive case, but with $X_{j}$ taken to be skew-Hermitian, so that
$Y_{j}=iX_{j}$ is Hermitian. Then we define the associated left-invariant
vector fields by%
\[
\hat{X}_{j}f(Z)=\left.  \frac{d}{du}f(Ae^{uX_{j}})\right\vert _{u=0};\quad
\hat{Y}_{j}f(Z)=\left.  \frac{d}{du}f(Ae^{uY_{j}})\right\vert _{u=0}.
\]
Then we construct \textquotedblleft multiplicative\textquotedblright%
\ operators $\Delta_{\mathrm{m}},$ $\partial_{\mathrm{m}}^{2},$ and
$\bar{\partial}_{\mathrm{m}}^{2}$ by replacing $\tilde{X}_{j}$ and $\tilde
{Y}_{j}$ by $\hat{X}_{j}$ and $\hat{Y}_{j}$ in (\ref{threeOps}). Then the
density of the law of $B_{s,\tau}^{N}$ is a heat kernel measure with Laplacian
equal to $s\Delta_{\mathrm{m}}-\tau\partial_{\mathrm{m}}^{2}-\bar{\tau}%
\bar{\partial}_{\mathrm{m}}^{2}$ \cite[Lemma 5.9]{DHKcomplex}, and these three
operators commute \cite[Corollary 5.7]{DHKcomplex}.

The multiplicative version of the identity (\ref{diffExpect}) in Lemma
\ref{firstID.lem} then follows from the second proof of Lemma
\ref{firstID.lem}, using (the general version of) Proposition 4.7 in
\cite{DHKcomplex}. (A differential equation for the density of the law of
$B_{s,\tau}^{N}$---as in (\ref{gammaPDE}) in the additive case---is not needed
for the proof of Theorem \ref{deformationMult.thm}, as long (\ref{diffExpect}) holds.)

We then state and prove the multiplicative version of Lemma \ref{secondID.lem}%
. Once this result is established, the proof of Theorem
\ref{deformationMult.thm} proceeds as in the additive case.

\begin{lemma}
\label{secondIDmult.lem}Let $B$ be a variable ranging over $M_{N}(\mathbb{C}%
)$, let $z$ be a variable ranging over $\mathbb{C}$, and let $A$ be a fixed
matrix. Then%
\[
\partial_{\mathrm{m}}^{2}\det(zI-AB)=\frac{1}{N}\left\{  z^{2}\frac
{\partial^{2}}{\partial z^{2}}-z(N-2)z\frac{\partial}{\partial z}-N\right\}
\det(zI-AB),
\]
where $\partial_{\mathrm{m}}^{2}$ acts in the $B$ variable.
\end{lemma}

\begin{proof}
We use the notation
\begin{align*}
Q  &  =zI-AB\\
R  &  =Q^{-1}.
\end{align*}
The multiplicative version of the operators $Z_{j}$ satisfy the basic identity%
\[
Z_{j}B=BX_{j}.
\]
Then we compute%
\begin{align*}
Z_{j}\det(Q)  &  =-N\det(Q)\mathrm{tr}[RABX_{j}]\\
Z_{j}^{2}\det(Q)  &  =N^{2}\det(Q)\mathrm{tr}[RABX_{j}]\mathrm{tr}[RABX_{j}]\\
&  -N\det(Q)\mathrm{tr}[RABX_{j}RABX_{j}]\\
&  -N\det(Q)\mathrm{tr}[RABX_{j}^{2}]
\end{align*}
Using the magic formulas (\ref{magic1}) and (\ref{magic2}), but this time with
the sign reversed because the $X_{j}$'s and $Y_{j}$'s are skew Hermitian, we
get%
\begin{align*}
\partial_{\mathrm{m}}^{2}\det(Q)  &  =-\det(Q)\mathrm{tr}[RABRAB]+N\det
(Q)\mathrm{tr}[RAB]\mathrm{tr}[RAB]\\
&  +N\det(Q)\mathrm{tr}[RAB]
\end{align*}

If we write%
\[
AB=zI-(zI-AB)=zI-Q,
\]
we get%
\[
\mathrm{tr}[RAB]=z\mathrm{tr}[RAB]-1
\]
and%
\[
\mathrm{tr}[RABRAB]=1-2z\mathrm{tr}[R]+z^{2}\mathrm{tr}[R^{2}]
\]
and
\[
\mathrm{tr}[RAB]\mathrm{tr}[RAB]=1-2z\mathrm{tr}[R]+z^{2}\mathrm{tr}[R]^{2}.
\]
Thus,%
\begin{align*}
\partial_{\mathrm{m}}^{2}\det(Q)  &  =\det(Q)\{(N-1)-2z(N-1)\mathrm{tr}[R]\}\\
&  +\det(Q)z^{2}\{N\mathrm{tr}[R]^{2}-\mathrm{tr}[R^{2}]\}\\
&  +N\det(Q)(-1+z\mathrm{tr}[R]),
\end{align*}
which simplifies to%
\begin{align*}
\partial_{\mathrm{m}}^{2}\det(Q)  &  =\det(Q)\{-1-z(N-2)\mathrm{tr}[R]\}\\
&  +\det(Q)z^{2}\{N\mathrm{tr}[R]^{2}-\mathrm{tr}[R^{2}]\}.
\end{align*}

Using (\ref{dzDet}) and (\ref{dzSquaredDet}), we obtain the claimed result.
\end{proof}

\subsection{The dynamics of the roots\label{multDynamics.sec}}

We now describe the multiplicative version of the results of Section
\ref{dynAdd.sec}.

\begin{proposition}
\label{multODE.prop}Let $p_{0}^{N}$ be a polynomial of degree $N$ and define
$p_{\tau}^{N}$ by
\[
p_{\tau}^{N}(z)=\exp\left\{  -\frac{\tau}{2N}\left(  z^{2}\frac{\partial^{2}%
}{\partial z^{2}}-(N-2)z\frac{\partial}{\partial z}-N\right)  \right\}
p_{0}^{N}(z),
\]
where the exponential, as applied to $p_{0}^{N},$ is defined as a convergent
power series. Suppose that the zeros of $p_{\sigma}$ are nonzero and distinct
for some $\sigma\in\mathbb{C}.$ Then for all $\tau$ in a neighborhood of
$\sigma,$ it is possible to order the zeros of $p_{\tau}$ as $z_{1}%
(\tau),\ldots,z_{N}(\tau)$ so that each $z_{j}(\tau)$ is nonzero and depends
holomorphically on $\tau,$ and so that the collection $\{z_{j}(\tau
)\}_{j=1}^{N}$ satisfies the following system of holomorphic differential
equations:%
\begin{equation}
\frac{1}{z_{j}(\tau)}\frac{dz_{j}(\tau)}{d\tau}=\frac{1}{2N}\left[
1+\sum_{k\neq j}\frac{z_{j}(\tau)+z_{k}(\tau)}{z_{j}(\tau)-z_{k}(\tau
)}\right]  . \label{dzjMult}%
\end{equation}
Furthermore, if we write each $z_{j}(\tau)$ as $z_{j}=e^{iw_{j}(\tau)}$ (but
where we do not assume $w_{j}(\tau)$ is real), we have the second-derivative
formula%
\begin{equation}
\frac{d^{2}w_{j}}{d\tau^{2}}=-\frac{1}{4N^{2}}\sum_{k\neq j}\frac{\cos
((w_{j}-w_{k})/2)}{\sin^{3}((w_{j}-w_{k})/2)}. \label{CMtrig}%
\end{equation}

\end{proposition}

\begin{remark}
\label{circleroots.remark}The rescaled roots $\tilde{z}_{j}(t)=e^{-\frac
{t}{2N}}z_{j}(t)$ satisfy a system of differential equations identical to
(\ref{dzjMult}) except with out the \textquotedblleft$1$\textquotedblright%
\ term on the right-hand side of the equation. Now, if $z$ and $w$ are in the
unit circle, then $(z+w)/(z-w)$ is pure imaginary. Using this observation, and
interpreting the left-hand side of (\ref{dzjMult}) as the derivative of $\log
z_{j}(\tau),$ we can verify the following result: If $\tau_{0}=0$ and the
points $\{z_{j}(0)\}_{j=1}^{N}$ are all in the unit circle, then the points
$\tilde{z}_{j}(t)$ will be remain in the unit circle for $t\in\mathbb{R}$, for
as long as they remain distinct. In the case $t>0,$ however, we are interested
in going well past the time when the points collide.

The system of differential equations for the points $\tilde{z}_{j}(t)$ in the
previous paragraph is the same as the one in Tao's blog post \cite{Tao2}, up
to a scaling of the time variable.
\end{remark}

As for the corresponding result (\ref{secondDeriv}) in the additive case, we
expect that the right-hand side of (\ref{CMtrig}) will be small when $N$ is
large. Thus, we expect that the trajectories $w_{j}(\tau)$ will be
approximately linear in $\tau,$ at least for small $\tau.$

\begin{remark}
\label{cmMult.remark}The equation (\ref{CMtrig}) is the equation of motion for
the \textbf{trigonometric Calogero--Moser} system, introduced by Sullivan.
(Take $a=1/2$ and $g^{2}=-1/(2N^{2})$ in Eq. (9) of \cite{Cal}.)
\end{remark}

\begin{proof}
[Proof of Proposition \ref{multODE.prop}]The verification of (\ref{dzjMult})
is very similar to the verification of (\ref{theODE}) in Proposition
\ref{ODEadd.prop} and is omitted. For (\ref{CMtrig}), we first compute that%
\begin{equation}
\frac{d}{d\tau}\frac{z_{j}+z_{k}}{z_{j}-z_{k}}=2\frac{z_{j}\frac{dz_{k}}%
{d\tau}-z_{k}\frac{dz_{j}}{d\tau}}{(z_{j}-z_{k})^{2}}. \label{zPlusZminus}%
\end{equation}
Then since $w_{j}=\frac{1}{i}\log z_{j},$ we use (\ref{zPlusZminus}) and
(\ref{dzjMult}) to compute%
\begin{align}
\frac{d^{2}w_{j}}{d\tau^{2}}  &  =\frac{1}{i}\frac{d}{d\tau}\left(  \frac
{1}{z_{j}}\frac{dz_{j}}{d\tau}\right) \nonumber\\
&  =\frac{1}{2N^{2}}\sum_{k\neq j}\frac{1}{(z_{j}-z_{k})^{2}}z_{j}%
z_{k}\nonumber\\
&  \cdot\left(  \sum_{l\neq k}\frac{z_{k}+z_{l}}{z_{k}-z_{l}}-\sum_{l\neq
j}\frac{z_{j}+z_{l}}{z_{j}-z_{l}}\right)  . \label{ddLog}%
\end{align}
Then, as in the proof of Proposition \ref{ODEadd.prop}, we write each sum over
$l$ in (\ref{ddLog}) as a sum over $l\notin\{j,k\}$ plus one extra term,
giving%
\begin{align*}
\frac{d^{2}w_{j}}{d\tau^{2}}  &  =\frac{1}{2iN^{2}}\sum_{k\neq j}\frac
{1}{(z_{j}-z_{k})^{2}}z_{j}z_{k}\\
&  \cdot\left(  -2\frac{(z_{j}+z_{k})}{(z_{k}-z_{k})}+\sum_{l\notin%
\{j,k\}}\left(  \frac{z_{k}+z_{l}}{z_{k}-z_{l}}-\frac{z_{j}+z_{l}}{z_{j}%
-z_{l}}\right)  \right)  .
\end{align*}
This result simplifies to
\begin{align}
\frac{d^{2}w_{j}}{d\tau^{2}}  &  =-\frac{1}{iN^{2}}\sum_{k\neq j}\frac
{z_{j}z_{k}(z_{j}+z_{k})}{(z_{j}-z_{k})^{3}}\nonumber\\
&  +\frac{1}{2iN^{2}}z_{j}\sum_{\substack{k,l:~\\(j,k,l)\text{ distinct}%
}}\frac{z_{k}z_{l}}{(z_{k}-z_{l})(z_{j}-z_{k})(z_{j}-z_{l})}.
\label{secondDerivLog}%
\end{align}

Now, the second term on the right-hand side of (\ref{secondDerivLog}) is zero
because the summand changes sign under interchange of $k$ and $l$. For the
first term, we recall that $z_{j}=e^{ix_{j}}$ and compute that%
\begin{equation}
\frac{z_{j}z_{k}(z_{j}+z_{k})}{(z_{j}-z_{k})^{3}}=\frac{i}{4}\frac{\cos
((x_{j}-x_{k})/2)}{\sin^{3}((x_{j}-x_{k})/2)}. \label{zjID}%
\end{equation}
Dropping the second term on the right-hand side of (\ref{secondDerivLog}) and
using (\ref{zjID}) in the first term gives (\ref{CMtrig}).
\end{proof}

\subsection{The PDE perspective}

In the multiplicative case, we again consider a sequence of polynomial $p^{N}$
with distinct roots, whose empirical root distribution converges to a
compactly supported probability measure $\nu_{0}$ on $\mathbb{C}$ whose Cauchy
transform is globally Lipschitz.

In this section, we propose the conjecture of the heat evolution of $p^{N}$
using a PDE perspective. We call this the multiplicative case because the PDE
is the same PDE as the one derived from multiplicative Brownian motion. The
argument why we believe the conjecture holds is very similar to the additive
case in Section \ref{sect.PDE.additive}.

\begin{conjecture}
Consider the new polynomial defined by
\[
q^{N}(z,\tau)=\exp\left\{  -\frac{\tau}{2N}\left(  z^{2}\frac{\partial^{2}%
}{\partial z^{2}}-(N-2)z\frac{\partial}{\partial z}-N\right)  \right\}
p^{N}(z).
\]
Let $\nu_{\tau}^{N}$ be the empirical root distribution of $p^{N}(\cdot,\tau
)$. Then the following holds.

\begin{enumerate}
\item There exists a constant $C>0$ such that $\nu_{\tau}^{N}$ converges
weakly to a probability measure $\nu_{\tau}$ on $\mathbb{C}$ whose logarithmic
potential $S(z,\tau) = \int\log\vert z-w\vert^{2} d\nu_{\tau}(w)$ is a $C^{1}%
$-solution of the PDE
\begin{equation}
\label{eq.HJ.limit.mult}\frac{\partial S}{\partial\tau} =- \frac{1}{2}\left(
z^{2}\left(  \frac{\partial S}{\partial z}\right)  ^{2}-z\frac{\partial
S}{\partial z}\right)  .
\end{equation}
for $\vert\tau\vert\leq C$.

\item There exists $0<\tau_{\mathrm{max}}<C$ such that the probability measure
$\nu_{\tau}$ on $\mathbb{C}$ is the push-forward of $\nu_{0}$ by
\[
T_{\tau}^{\mathrm{mult}}(z)=z\exp\left(  \tau m_{0}(z)-\frac{1}{2}\right)
\]
whenever $|\tau|\leq\tau_{\mathrm{max}}$, where $m_{0}$ denotes the Cauchy
transform of $\nu_{0}$.
\end{enumerate}
\end{conjecture}

\subsection{The evolution of the holomorphic moments}

We now briefly indicate how the results of Section \ref{moments.sec} need to
be modified in the multiplicative case.

\begin{theorem}
\label{holoMult.thm}The holomorphic moments in the multiplicative case satisfy
Theorem \ref{holoMoment.thm}, with the following changes.
First,(\ref{firstTwoMoments}) is replaced by
\[
\frac{dM_{0,N}}{d\tau}=0;\quad\frac{dM_{1,N}}{d\tau}=\frac{1}{2}M_{1,N}.
\]
Second, (\ref{dMdTauN}) is replaced by%
\begin{equation}
\frac{dM_{a,N}}{d\tau}=\frac{a}{2}M_{a,N}+\frac{a}{2}\sum_{b=1}^{a-1}%
M_{b,N}^{{}}M_{a-b,N}-\frac{a}{2N}M_{a,N}. \label{MomMult2}%
\end{equation}
Third, (\ref{dMdTau}) and (\ref{Mhat}) are replaced by the formal large-$N$
limit of (\ref{MomMult2}) (dropping the $1/N$ term). Last, in Point 3 of the
theorem, the PDE (\ref{addPDEintro}) is replaced by (\ref{eq.HJ.limit.mult}).
\end{theorem}

Note that the right-hand side of (\ref{MomMult2}) involves the moment
$M_{a,N}$. But we can still easily integrate (\ref{MomMult2}) to obtain
$M_{a,N}$ by using an exponential integrating factor. The proof of Theorem
\ref{holoMult.thm} is very similar to the proof of Theorem
\ref{holoMoment.thm} and is omitted.

\subsection*{Acknowledgments}

The authors thank the referees for reading the manuscript carefully and making many helpful suggestions for improvements.

\subsection*{Conflict of interest statement}

On behalf of all authors, the corresponding author states that there is no
conflict of interest.

\subsection*{Data availability statement}

This document has no associated data.

\end{document}